\documentclass[11pt]{article}

\usepackage{bbm}
\usepackage{amsthm}
\usepackage{amssymb}
\usepackage{amsmath}
\usepackage{color}
\usepackage{graphics}
\usepackage{graphicx}
\usepackage{hyperref}
\usepackage{caption}
\usepackage{tikz}
\usepackage{subcaption}
\usepackage{pgfplots}
\pgfplotsset{compat=newest}
\usepgfplotslibrary{fillbetween}
\usetikzlibrary{automata,arrows,positioning,calc}

\DeclareMathOperator*{\argmax}{arg\,max}

\newtheorem{theorem}{Theorem}
\newtheorem{assumption}{Assumption}
\newtheorem{proposition}{Proposition}
\newtheorem{lemma}{Lemma} 

\newtheorem{definition}{Definition}
\theoremstyle{remark}
\newtheorem{remark}{Remark}

\newcommand{\brac}[1]{\left( #1\right)}

\newcommand{\abs}[1]{\left| #1\right|}
\newcommand{\floor}[1]{\left\lfloor #1\right\rfloor}
\newcommand{\expect}[2][]{\mathbb{E}^{#1}\left[ {#2}\right]}
\newcommand{\prob}[1]{\mathbb{P}\left( #1\right)}

\newcommand{\norm}[1]{\left\lVert{#1}\right\rVert}
\newcommand{\indicator}[1]{ \mathbbm{1}_{\{#1\}} }

\newcommand{\sign}{\text{sign}}
\newcommand{\QED}{\null\hfill$\square$}

\newcommand{\floornt}{\lt\lfloor nt\rt\rfloor}
\newcommand{\sfW}{\mathsf{W}}

\newcommand{\sfC}{\mathsf{C}}

\newcommand{\sfY}{\mathsf{Y}}
\newcommand{\sfX}{\mathsf{X}}
\newcommand{\sfZ}{\mathsf{Z}}
\newcommand{\sfR}{\mathsf{R}}
\newcommand{\sfU}{\mathsf{U}}
\newcommand{\sfV}{\mathsf{V}}

\newcommand{\tot}{{\rm tot}}
\newcommand{\id}{{\tt id}}
\newcommand{\ttR}{\texttt{R}}
\newcommand{\ttW}{\text{\rm\texttt{W}}}
\newcommand{\vc}[1]{\mathbf{#1}}

\newcommand{\eps}{\varepsilon}
\newcommand{\ph}{\varphi}

\newcommand{\al}{\alpha}

\newcommand{\sig}{\sigma}
\newcommand{\del}{\delta}
\newcommand{\om}{\omega}
\newcommand{\Gam}{\mathnormal{\Gamma}}
\newcommand{\Del}{\mathnormal{\Delta}}

\newcommand{\PI}{\mathnormal{\Pi}}

\newcommand{\Om}{\mathnormal{\Omega}}

\newcommand{\B}{{\mathbb B}}

\newcommand{\N}{{\mathbb N}}

\newcommand{\R}{{\mathbb R}}

\newcommand{\Z}{{\mathbb Z}}

\newcommand{\EE}{{\mathbb E}}
\newcommand{\E}{{\mathbb E}}

\newcommand{\PP}{{\mathbb P}}

\newcommand{\calA}{{\cal A}}
\newcommand{\calB}{{\cal B}}
\newcommand{\calC}{{\cal C}}
\newcommand{\calD}{{\cal D}}
\newcommand{\calE}{{\cal E}}
\newcommand{\calF}{{\cal F}}
\newcommand{\calG}{{\cal G}}
\newcommand{\calH}{{\cal H}}

\newcommand{\calK}{{\cal K}}

\newcommand{\calP}{{\cal P}}

\newcommand{\calR}{{\cal R}}
\newcommand{\calS}{{\cal S}}

\newcommand{\bxi}{{\boldsymbol \xi}}
\newcommand{\bb}{{\boldsymbol b}}
\newcommand{\bq}{{\boldsymbol q}}
\newcommand{\bsigma}{{\boldsymbol \sigma}}

\newcommand{\frS}{\mathfrak{S}}

\newcommand{\uu}{\underline}

\newcommand{\skp}{\vspace{\baselineskip}}

\newcommand{\w}{\wedge}
\newcommand{\lt}{\left}
\newcommand{\rt}{\right}
\newcommand{\pl}{\partial}

\newcommand{\To}{\Rightarrow}

\newcommand{\iy}{\infty}
\newcommand{\noi}{\noindent}

\begin{document}

\title{
Scheduling in the high uncertainty heavy traffic regime}
\author{Rami Atar\thanks{Viterbi Faculty of Electrical and Computer Engineering, Technion, Haifa, Israel}
\and
Eyal Castiel${}^*$
\and Yonatan Shadmi${}^*$
}

\maketitle

\begin{abstract}
We propose a model uncertainty approach to
heavy traffic asymptotics that allows for a high level of uncertainty.
That is, the uncertainty classes
of underlying distributions accommodate disturbances
that are of order 1 at the usual diffusion scale,
as opposed to asymptotically vanishing disturbances studied previously
in relation to heavy traffic.
A main advantage of the approach is that
the invariance principle underlying diffusion limits
makes it possible to define uncertainty classes in terms of the first two moments only.
The model we consider is a single server queue with multiple job types.
The problem is formulated as a zero sum stochastic game played
between the system controller, who determines scheduling
and attempts to minimize an expected linear holding cost,
and an adversary, who dynamically controls the service time distributions
of arriving jobs, and attempts to maximize the cost.
The heavy traffic asymptotics of the game are fully solved.
It is shown that an asymptotically optimal policy for
the system controller is to prioritize according to
an index rule and for the adversary it is to select distributions
based on the system's current workload.
The workload-to-distribution feedback mapping
is determined by an HJB equation,
which also characterizes the game's limit value.
Unlike in the vast majority of results in the heavy traffic theory,
and as a direct consequence of the diffusive size disturbances,
the limiting dynamics under asymptotically optimal play
are captured by a stochastic differential equation
where both the drift and the diffusion coefficients may be
discontinuous.

\skp

\noi{\bf AMS subject classification:}
60K25; 93E20; 91A15; 60F17; 91A05; 68M20

\skp

\noi{\bf keywords:}
heavy traffic;
model uncertainty;
high uncertainty regime;
stochastic game;
HJB equation;
drift-variance tradeoff;
diffusion with discontinuous coefficients
\end{abstract}

\section{Introduction}
\label{sec1}

\subsection{Background, motivation and setting}

The asymptotic analysis of queueing systems under heavy traffic
and their non-asymptotic analysis under model uncertainty have
both been subject to extensive research.
There are very few papers on settings that combine the two;
those that we are aware of are \cite{coh19,coh19b,coh-sah} and \cite{sun2021}.
In this body of work, the effective size of uncertainty classes diminishes
as the heavy traffic scaling parameter grows.
The goal of this paper is to propose an approach
that combines heavy traffic and model uncertainty
in a way that keeps the disturbances associated with uncertainty
at order one under the usual diffusion scale.
As is well known, owing to the invariance principle governing diffusion limits,
results in the heavy traffic regime enjoy the robustness
property that the approximations they provide are based on only
the first two moments of the underlying distributions.
We shall argue that in the setting proposed here
the same reasons make it possible to
transform uncertainty classes in the space of distributions
to classes defined by the first two moments,
a desired property from model uncertainty viewpoint.

Although the approach is potentially applicable for a large variety of models,
we focus here on one model where a single server
caters to a number of streams of jobs of different types,
and a system controller (SC)
dynamically allocates the server's effort to the different streams.
The system is subject to model uncertainty
with respect to the service time distributions.
The model is often referred to as a {\it multiclass} queue,
but in this paper we reserve the term {\it class} to the context of uncertainty
and instead use the term {\it type} to distinguish between streams of jobs,
hence use {\it multitype} as substitute for {\it multiclass}.
To capture uncertainty, a stochastic zero sum game is formulated in which
an adversary can dynamically select the distributions
of arriving job sizes from given uncertainty classes, one per job type.
The cost consists of an expected discounted linear combination of queue lengths.
The asymptotics of the game are fully solved.
It is shown that it is asymptotically optimal (AO) for the SC
to prioritize according to the $c\mu$ rule regardless of the adversary behavior.
AO play for the adversary is derived
in terms of an underlying ordinary differential equation
of Hamilton-Jacobi-Bellman (HJB) type, which, moreover, characterizes the
value function asymptotics.

The importance of accounting for uncertainty in queueing models
is widely acknowledged. In what follows we mention a small sample
of work in this area.
An adversarial approach
to questions regarding stability of queueing networks
is proposed in \cite{bor96}. It is developed in \cite{gam00}
by exploiting the relation between queueing models and their
corresponding fluid models.
In \cite{ban15}, a worst case approach based on techniques of robust optimization
is proposed and used to obtain performance bounds on several queueing
models. In \cite{jai10}, a queueing control problem is solved
in which model uncertainty is characterized by relative entropy.
The modeling of call centers has motivated much work on model uncertainty,
including \cite{bas05, che-has, har-zee, koc15, whi06}.
Among these, \cite{bas05, har-zee, whi06} model uncertainty via stochastic
fluid models, where in particular, the fluid scale asymptotics are rigorously
justified in \cite{bas05}.
A staffing problem is solved in
\cite{che-has} in the Halfin-Whitt heavy traffic regime,
where a worst case approach is taken (among others)
to uncertainty in the arrival process.
Another control problem for a call center model is analyzed in
\cite{koc15}, where uncertainty is modeled by stochastic arrival rates
and diffusion-scale asymptotic optimality of a proposed control policy
is proved.
In \cite{ata02, dup03}, robust control problems associated with fluid queueing
models are considered via differential games.
A robust approach to queueing models at the large deviations regime
is proposed in \cite{ata21}.

The work that is most relevant to ours
is the aforementioned series \cite{coh19,coh19b,coh-sah}.
It combines an adversarial and an asymptotic approach in the heavy traffic regime,
and studies a queueing system similar to the one studied here.
In \cite{coh19,coh19b} buffers are finite,
cost is linear in queue length and rejection count,
whereas in \cite{coh-sah} buffers are infinite and cost is strictly convex in queue length.
The uncertainty class of arrival and service time distributions is based on
a reference model corresponding to the multitype M/M/1 queue,
in which arrival and potential service processes are Poisson with given intensities.
The adversary is allowed to perturb the law of these processes,
but is subject to a penalty expressed in terms of a functional which is a
generalization of the Kullback-Leibler
divergence between the perturbed and reference measures.
The results of \cite{coh19,coh19b} characterize
the limiting value function in terms of an HJB equation and construct
AO strategies, whereas the paper \cite{coh-sah} proves that the generalized $c\mu$ rule is an AO
strategy for the SC and provides an explicit solution
to a differential game which governs the asymptotics of the problem.
Similar setting and methodology appear also in the preprint \cite{sun2021}.

Our treatment differs from that in \cite{coh19,coh19b,coh-sah}
in a number of ways, two of which are crucial:
The collection of distributions allowed in the uncertainty class,
and the asymptotic size of this class.
To explain the first aspect, note that because in \cite{coh19,coh19b,coh-sah}
model uncertainty is quantified in terms of penalty for deviations measured
by divergence, the uncertainty classes of the underlying counting processes
are automatically restricted to absolutely continuous
changes of measure with respect to Poisson processes.
However, standard (non-game)
heavy traffic limit theorems are typically concerned
with much larger classes of distributions
within the domain of applicability of CLT approximations,
namely those restricted only by moment assumptions.
We aim here at uncertainty classes defined only in terms of moments,
and not restricted to absolutely continuous changes of measures of one another.

The second crucial aspect has to do with the scale of the uncertainty classes.
In the setting of \cite{coh19,coh19b,coh-sah},
the effective size of the uncertainty classes shrinks like $n^{-1/2}$ as $n$ grows,
where we have denoted by $n$ the usual heavy traffic scaling parameter.
More precisely, it is shown in these references that
there is no loss for the adversary to consider only changes of measures
for which the stochastic intensities of the underlying counting processes
are $(1+O(n^{-1/2}))$-multiple of the reference model's Poissonian intensity
(see Proposition 4.4 of \cite{coh19} and Proposition 4.3 of \cite{coh-sah}).
For example, if the potential service process of type 1 jobs is Poisson of rate $\mu_1$
under the reference measure then within the uncertainty class
it has stochastic intensity bounded
above and below by $\mu_1(1+O(n^{-1/2}))$, and consequently
the perturbed service time distributions are bounded between
${\rm Exp}(\mu_1+cn^{-1/2})$ and ${\rm Exp}(\mu_1-cn^{-1/2})$
in the sense of usual stochastic ordering
(here, ${\rm Exp}(\mu)$ denotes the exponential distribution with parameter
$\mu$).
It is proposed to refer to this setting as the {\it low uncertainty regime}.
Our aim is to address a different natural setting in which
the perturbations are of order $1$, in a sense to be made precise,
to which we shall refer as the {\it high uncertainty regime}.
One should not regard one regime superior to the other;
rather, the two complement each other,
each capturing realistic modeling assumptions under different scenarios.

We now give some more details on the setting in which we work.
The model can be regarded as a multitype D/G/1 queue
with uncertainty.
The focus is on stochasticity associated only with sizes of jobs,
not with their arrival times, that are deterministic by assumption.
Although it is desired to model arrival stochasticity,
avoiding it saves much notational burden, whereas
mathematically this aspect is less difficult.
However, because jobs of size zero are allowed to occur with
positive probability, the setting automatically captures
exponentially distributed interarrival times.

To fix some notation, consider for the moment
the much simpler, single D/G/1 queue with fixed underlying
law (i.e., no uncertainty), described as follows.
For each $n\in\N$, let $\{J^n_k\}_{k\in\N}$ be IID nonnegative RVs.
In the $n$-th system, an arrival occurs every unit of time, and the
$k$-th job to arrive is of size $J^n_k$. Denote $b_n=n^{-1/2}(E[J^n_1]-1)$
and $\sig_n^2={\rm var}(J^n_1)$. It is well known that
if $(b_n,\sig_n^2)\to(b,\sig^2)\in\R\times(0,\iy)$ then
under Lindberg-Feller conditions and convergence of initial conditions,
the diffusion scaled workload of the $n$-th system converges, as $n\to\iy$,
to a reflected Brownian motion (BM) with infinitesimal drift and variance
$b$ and $\sig^2$.
For this reason we shall refer to $b_n$ and $\sig_n^2$ as the {\it prelimit
drift and variance}.
Similarly, in a multitype setting, there is a pair of prelimit coefficients
for each type.

We can now explain how the uncertainty classes are defined
in this paper. For each $n$ and type $\ell$, an uncertainty class is a set
in the space of measures on $\R_+$, representing probability laws of
job sizes. For each such measure there corresponds a pair $(b_n,\sig_n^2)$
defined analogously to the above example. The collection of pairs $(b_n,\sig_n^2)$
corresponding to all measures in the uncertainty class is denoted by
$K^{\ell n}\subset \R^2$. It is assumed that, as $n\to\iy$,
$K^{\ell n}\to K^\ell$ in the Hausdorff metric, where
$K^\ell$ is a compact (but otherwise arbitrary) subset of $\R\times(0,\iy)$.
A certain uniform integrability condition is also assumed.

According to this definition,
an uncertainty class may contain distributions for which
the second moment differs by $O(1)$.
The same is true, for example, for the absolute centered first moment.
The only sense in which the distributions within the
class must become close to each other
is in the first moment, where they can differ by only $O(n^{-1/2})$.
This is necessary in order to not trivialize the problem:
if the adversary is allowed to choose between
prelimit drifts $b_n$ that are apart by more than $O(n^{-1/2})$
then only those with greatest $b_n$ will be chosen,
making the problem much simpler.

\subsection{Results}

A first key step in our analysis is to relate the workload process,
whose limit provides the state process of the dynamics, to the
queue length, in terms of which the cost is expressed.
In non-game heavy traffic settings, these processes are asymptotically
proportional to one another, a phenomenon
known as Reiman's snapshot principle (RSP) \cite{rei82}.
In order to analyze the game one needs to develop, as we do,
a version of RSP in which
the asymptotic proportionality is attained uniformly in the actions of the SC.

The $c\mu$ rule is well known to be AO in the context of linear queue length
cost. Our first main result is based on the aforementioned RSP and
states that this rule defines an AO policy for the SC under
arbitrary behavior of the adversary.
Thanks to this result, we can, and do, analyze the game assuming
that the $c\mu$ rule is always used by the SC.
Thus, although the problem we are set to solve is concerned with
game asymptotics, its limit is reduced to a stochastic control problem
that involves only the adversary, which is considerably simpler than
a stochastic differential game.

In this control problem, the state process is a one-dimensional controlled
diffusion with reflection at the origin and controlled drift and diffusion coefficients.
It is well known that solutions to a control problem of this
kind are characterized in terms of viscosity solutions of an HJB equation,
but classical solutions do not always exist.
In our case we show by appealing to PDE theory on fully nonlinear uniformly
elliptic equations that the HJB equation has a unique classical solution.
The solution to the control problem is given by a diffusion
with discontinuous drift and diffusion coefficients.

Our second main result provides a full solution to the game asymptotics.
It states that the stochastic game's value
converges to that of the diffusion control problem
or, equivalently, to the solution of the HJB equation.
It also specifies an AO policy for the adversary.
It is established in its proof that any subsequential limit of the game dynamics
is given by the aforementioned diffusion with discontinuous
coefficients; in this respect
the limit result resembles that in the recent work \cite{ACR-LB, ACR}.

Finally we state a result that provides insight into the
collection of distributions within the uncertainty classes
from which the adversary selects. As we will explain, it is standard to deduce
from the form of the HJB equation that an adversary that
plays optimally will select distributions only from a part of the boundary
of the set, specifically, the set that generates the convex hull of $K$.
Our results reduces further the collection of potentially used
points in the class, to what we call the set of dominating points.
This is the collection of points on the boundary
that are dominating the whole set
w.r.t.\ the partial order $(b,\sig)\le(b',\sig')$ iff
$b\le b'$ and $\sig\le\sig'$.

\subsection{Organization of the paper}

\S \ref{sec2} describes the queueing model, the game and the form of uncertainty classes,
and introduces a martingale control problem (MCP) which is later proved to govern the game asymptotics.
It then states Theorem \ref{thm:cmu is opt} on the AO of the $c\mu$-rule,
Proposition \ref{pr:verification}, which relates the MCP value to an HJB equation,
and Theorem \ref{thm:Vn to V}, which establishes a relation between the game asymptotics,
the MCP and the HJB equation. \S \ref{sec3} contains remarks and examples.
The remaining sections contain proofs. \S \ref{sec:cmu} provides the proof of
Theorem \ref{thm:cmu is opt}. \S \ref{sec:mcp} studies the MCP and the HJB equation
and proves Proposition \ref{pr:verification}. Finally, Theorem \ref{thm:Vn to V} is proved
in \S\S\ref{sec:ub}--\ref{sec:ao}, where the former establishes a general upper bound
on the game asymptotics in terms of the MCP value, and the latter identifies
a sequence of controls for the adversary that asymptotically achieves this bound,
thereby establishing the convergence of the game's value.

\subsection{Notation}

For $k\in\N$, $[k]=\{1,\ldots,k\}$. $\R_+=[0,\iy)$. 
For $u,v\in\R^k$, $u\cdot v=\sum_i u_i v_i$.
For $a,b\in\R$, $a\le b$, $\sum_{i=a}^b=\sum_{i\in\Z\cap[a,b]}$.
For $a, b \in \R$, $a \vee b$ and $a\w b$ denote the maximum
and, respectively, minimum of $a$ and $b$, and $a^+=a \vee 0$.
For a metric space $(M,d)$,
the Hausdorff metric between two nonempty sets $U,V\subset M$
is defined by
\[
d_{\rm H}(U,V)= \sup\{d(x,V): x\in U\}\vee\sup\{d(y,U): y\in V\},
\]
where $d(x,U)$ denotes the distance between a point $x$ and a set $U$.
For a subset $K\subset\R^k$, $\text{ch}(K)$ denotes its closed convex hull.
The symbol $\id:\R_+\to\R_+$ denotes the identity map.
For $f:\R_+\to\R$ and $t,\del>0$, denote
$\|f\|_t=\sup_{s\in[0,t]}|f(s)|$ and
\[
w_t(f,\del)=\sup\{|f(s_1)-f(s_2)|:0\le s_1\le s_2\le(s_1+\del)\w t\}.
\]
For $0\le s\le t$, the notation $f[s,t]$ stands for $f(t)-f(s)$
and $\Del f_t=f(t)-f(t-)$ (when the left-limit exists).
For real-valued functions and processes, the notation $X(t)$ is
used interchangeably with $X_t$.
Given a Polish space $E$, denote by $\calC_E[0,\iy)$ and $\calD_E[0,\infty)$
the spaces of $E$-valued, continuous and, respectively, c\`{a}dl\`{a}g functions on $[0,\infty)$.
In the case $E=\R$ simply denote $\calC$ and $\calD$ respectively.
Equip the former with the topology of convergence u.o.c.\ and the latter with the $J_1$ topology.
Denote by $\calC^+_{\R^k}[0,\iy)$ (respectively, $\calD^+_{\R^k}[0,\infty)$)
the subset of $\calC_{\R^k}[0,\iy)$
(respectively, $\calD_{\R^k}[0,\iy)$) of componentwise
non-negative and non-decreasing functions.
Write $X_n\To X$ for convergence in law.
A tight sequence of processes with sample paths in $\calD_E[0,\iy)$ is said to be $\calC$-tight
if it is tight and the limit of every
weakly convergent subsequence has sample paths in $\calC_E[0,\iy)$ a.s.
The letter $c$ denotes a deterministic constant whose value may change from one
appearance to another.

\section{Model and results}
\label{sec2}

\subsection{Queueing model and game setting}\label{sec21}

\subsubsection{Queueing model}

The queueing model is that of a multitype D/G/1 queue with $L$ types of jobs,
operating in continuous time, where each type has a dedicated unlimited buffer, and arrivals
into each of these buffers occur at deterministic times.
In the $n$-th system an arrival occurs every $n^{-1}$ units of time
(starting at time $n^{-1}$), and the service rate is $n$.
All random variables (RVs) and processes introduced below
are defined on a probability space $(\Om,\calF,\PP)$.
For $k\in\N$, at each time $kn^{-1}$, a type-$\ell$ job
of size $J_k^{\ell n}$ arrives  at the $\ell$-th buffer of the $n$-th system.
As a result, for each $\ell$,
the number of type-$\ell$ arrivals by time $t$, and the work associated with these arrivals
are given, respectively, by
\begin{align}\label{eq:Al}
\lfloor nt\rfloor,
\qquad
A^{\ell n}_t=\sum_{k=1}^{nt}J_k^{\ell n},
\qquad t\in\R_+.
\end{align}
We note that we do not regard \eqref{eq:Al} as rescaling of time by factor $n$, rather $nt$ is the number of jobs arriving by time $t$.
Henceforth, a superscript `$\tot$' (for `total')
will denote summation over $\ell\in[L]$,
e.g.\ $A^{\tot,n}_t=\sum_{\ell=1}^LA_t^{\ell n}$, etc.

Next we introduce the control process for the SC, denoted by
$B^n_t=(B^{\ell n}_t)_{\ell\in[L]}$. In the special case where the server's behavior
is to devote all its effort to one type at a time,
$B^{\ell n}_t$ is simply the time the server has devoted to type-$\ell$ jobs
by time $t$. In general, resource sharing is allowed and therefore
$B^{\ell n}_t$ represents the cumulative effort provided
to $\ell$-type jobs by time $t$. Because the server works at rate $n$,
the work done by time $t$ is given by $nB^{\ell n}_t$, for each type $\ell$.
Denote the simplex in
$\R^L$ by ${\rm Sim}=\{x\in[0,1]^L: \sum_\ell x_\ell\le 1\}$.
Then the sample paths of $B^n$ take values in
\[
\uu{\calB}:={\rm Lip}({\rm Sim}):=
\{\psi\in \calC^+_{\R^L}[0,\iy):\psi(0)=0,\, (t-s)^{-1}(\psi(t)-\psi(s))\in{\rm Sim}
\text{ for all } 0\le s<t\}.
\]

It is assumed that within each type, jobs are served according to FIFO.
It is also assumed that the system starts empty.
Let $W^n_t=(W^{\ell n}_t)_{\ell\in[L]}$ denote the
workload at the different buffers at time $t$. Then
this process is uniquely determined by the workload arrival process
$A^n_t$ and the cumulative effort process $B^n_t$ via
\begin{equation}\label{eq:Wl}
    W^{\ell n}_t=A^{\ell n}_t-nB^{\ell n}_t\ge0,
    \qquad \ell\in[L],\ t\ge0.
\end{equation}
Above, we have also expressed the nonnegativity of the workload,
which we shall regard as a constraint imposed
on the SC when selecting the process $B^n$.
The term $nB^{\ell n}_t$ reflects the fact that
the service rate is given by $n$.

Next, the queue length $Q^{\ell,n}_t$ can also be determined
from $A^n$ and $B^n$. To express it, let
\begin{align*}
    &S^{\ell n}_k=\inf\{t\geq 0:nB^{\ell n}_{t}\ge A^{\ell n}_{k/n}\},\\
    &D^{\ell n}_t=\sup\{k\geq 0:S^{\ell n}_k\le t\}.
\end{align*}
Then $S^{\ell n}_k$ represents the time the $k$-th $\ell$-type job
has left the system, and $D^{\ell n}_t$, the departure process,
represents the number of $\ell$-type jobs that have left the system
by time $t$.
Then we have
\begin{align}\label{eq:Ql}
Q^{\ell n}_t=\lfloor nt \rfloor-D^{\ell n}_t.
\end{align}
Note that by their construction, the processes
$A^n$, $W^n$, $Q^n$ have sample paths in $\calD_{\R_+^L}[0,\iy)$.

The diffusion scale versions of $W^n$,  $Q^n$ and $R^n$ are given by
$\hat W^n=n^{-1/2}W^n$, $\hat Q^n=n^{-1/2}Q^n$ and $\hat R^n=n^{-1/2}R^n$.
Several additional diffusion scaled processes are defined later along with
conditions for heavy traffic.

\subsubsection{Game formulation}

We now describe a game, defined for each $n$,
played between the SC that controls $B^n$
and an adversary that controls the laws from which $J^n$ are drawn.
To introduce the cost, let a vector $h\in(0,\iy)^L$ be given, and let
\begin{align}\label{27}
    C^n(B^n,J^n)=\E\Big[\int_0^\infty e^{-t}h\cdot \hat{Q}^n_tdt\Big],
    \qquad n\in\N.
\end{align}
Above, $\hat Q^n$ is the process determined by $B^n$ and $J^n$.
The SC's goal is to minimize this cost, whereas the adversary, whose goal is to
disturb this effort as much as possible, attempts to maximize it.
The precise details are as follows.

\paragraph{Adversary controls.}

Let $\calP$ denote the space of probability measures on $\R_+$
endowed with the topology of weak convergence.
For each $n$ and $\ell$, a nonempty closed set $\PI^{\ell n}\subset\calP$ is given,
playing the role of an uncertainty class for type-$\ell$ job size distributions.
The product $\PI^{1n}\times\cdots\times\PI^{L n}$ is denoted by $\PI^n$.
At each arrival time, $n^{-1}k$, the adversary selects the random vector
of job sizes, $J^n_k=(J^{\ell n}_k)_{\ell\in[L]}$, in two steps.
First, it selects for each type $\ell$ a distribution from the corresponding
uncertainty class, namely
$\pi^{\ell n}_k\in\PI^{\ell n}$. Before doing so it may observe all
past job sizes $J^n_1,\ldots,J^n_{k-1}$. This is expressed by requiring
\begin{equation}\label{20}
\pi^n_k=(\pi^{\ell n}_k)_{\ell\in[L]}\in\calF^{n}_{k-1}
:=\sig\{J^n_m:m\le k-1\}.
\end{equation}
We also denote
\begin{equation}\label{30}
\calG^n_t=\calF^n_{\floor{nt}},\qquad t\ge0,
\end{equation}
and note that $\calF^n_k=\calG^n_{k/n}$. 

In the second step, the adversary selects the random vector $J^n_k$
so that it follows the product law corresponding to these newly selected distributions.
More precisely, for $C_\ell\in\calR_+$ (Borel subsets of $\R_+$), $\ell\in[L]$,
\begin{equation}\label{22}
\PP(\cap_\ell \{J^{\ell n}_k\in C_\ell\}|\calF^{n}_{k-1})
=\prod_\ell\pi^{\ell n}_k(C_k).
\end{equation}
(We assume without loss of generality that the probability space supports such RVs).

An admissible control for the adversary is thus a sequence
$(a^n_k)_{k\in\N}=(\pi^n_k,J^n_k)_{k\in\N}$ of $\PI^n\times\R_+^L$-valued RVs,
satisfying \eqref{20} and \eqref{22}. The collection of all admissible
controls for the adversary is denoted by $\calA^n$.

\begin{remark}
An alternative way to define the second step is to set $\calA^n=\PI^n$ and let
$J^n_k$ be selected arbitrarily according to \eqref{22} rather than
letting the adversary select it. In this case one must show that the way in
which it is selected always leads to the same game.
This is true but requires some further arguments, which we have
chosen to omit by letting the adversary select $J^n_k$
(which also leads to the same game).
Moreover, it may at first seem that when selecting the pair
$(\pi^n_k,J^n_k)$, the adversary can compare outcomes of
different pairs and select the best disturbance among these outcomes.
However, this would violate
the requirement \eqref{20} that $\pi^n_k$ is selected based only on
the past outcomes, $J^n_1,\ldots,J^n_{k-1}$.
\end{remark}

\paragraph{SC strategies.}
In this paragraph we use underline to denote members of the space of
sample paths of a corresponding stochastic processes.
In particular, let $\uu{\calA}^n=(\PI^n\times\R^L_+)^\N$. Then
$\uu{a}^n=(\uu{\pi}^n,\uu{J}^n)\in\uu{\calA}^n$ is deterministic,
and $a^n=(\pi^n,J^n)\in\calA^n$ defined above is a process that has sample paths
in $\uu{\calA}^n$.

We topologize $\uu{\calA}^n$ with the product topology and recall that $\uu{\calB}$ is equipped with the topology of u.o.c. convergence;
we consider both spaces with their Borel $\sig$-fields.
A measurable map
$\beta^n:\uu{\calA}^n\to\uu{\calB}$ is said to be a {\it strategy}
for the SC in the $n$-th system. By selecting a strategy $\beta^n$, the SC
determines the control process $B^n$ as a response to
the adversary's control $a^n$ via
\begin{equation}\label{26}
B^n(\om)=\beta^n(a^n(\om)), \qquad \om\in\Om.
\end{equation}
For a strategy $\beta^n$ to be admissible, it must satisfy several additional
properties.

{\it - Nonnegativity constraint.}
Given $n$, if $\uu{a}^n=(\uu{\pi}^n,\uu{J}^n)\in\uu{\calA}^n$
and $\uu{B}^n=\beta^n(\uu{a}^n)$, then for the nonnegativity constraint
\eqref{eq:Wl} to hold we must require that
\begin{equation}\label{23}
\uu{A}^n_t-n\uu{B}^n_t\in\R^L_+ \text{ for all $t$, where }
\uu{A}^n_t=\sum_{k=1}^{nt}\uu{J}^n_k.
\end{equation}

{\it - Work conservation.}
We will only consider work conserving strategies.
For the precise details we need to introduce two pieces of notation.
First, let $R^n_t$ denote the idleness process
$R^n_t=nt-nB^{\tot,n}_t$.
Work conservation is the property that
the server works at full capacity whenever there is work in the system.
This can be expressed mathematically as
$\int_{[0,\iy)}W^{\tot,n}_tdR^n_t=0$.

Second, enter the {\it Skorohod map} on the half line.
This map, denoted throughout this paper by
$\Gam:\calD_{\R}[0,\iy)\to \calD_{\R_+}[0,\iy)\times \calD^+_{\R}[0,\iy)$,
sends a function $\psi$ to a pair $(\ph,\eta)$, where
\begin{equation}\label{e02}
\ph(t)=\psi(t)+\eta(t),\qquad \eta(t)=\sup_{0\le s\le t}\psi(s)^-,
\qquad t\ge0.
\end{equation}
The corresponding maps $\psi\mapsto\ph$ and $\psi\mapsto\eta$ are denoted
by $\Gam_1$ and $\Gam_2$, respectively.
{\it Skorohod's lemma} states that for $\psi\in \calD_\R[0,\iy)$,
if $(\ph,\eta)\in \calD_{\R_+}[0,\iy)\times \calD^+_{\R}[0,\iy)$, $\ph=\psi+\eta$ and
$\int_{[0,\iy)}\ph_td\eta_t=0$ then $(\ph,\eta)=\Gam(\psi)$.

Putting together the last two comments, noting that by \eqref{eq:Wl}
\[
W^{\tot,n}=A^{\tot,n}-nB^{\tot,n}=A^{\tot,n}-n\id+R^n,
\]
it follows that 
\begin{align}\label{eq:WR=Gamma}
    (W^{\tot,n},R^n)=\Gam(A^{\tot,n}-n\,\id).
\end{align}
Hence work conservation is expressed by requiring that the map $\beta^n$ satisfies the following.
Given $n$, if $\uu{a}^n=(\uu{\pi}^n,\uu{J}^n)\in\uu{\calA}^n$
and $\uu{B}^n=\beta^n(\uu{a}^n)$ then
$n\,\id-n\uu{B}^{n,\rm tot}=\Gam_2(\uu{A}^{\tot,n}-n\,\id)$, namely
\begin{equation}\label{24}
\uu{B}^{\tot,n}=\id-\Gam_2(n^{-1}\uu{A}^{\tot,n}-\id).
\end{equation}

{\it - Causality.}
We shall also require admissible strategies to be causal (or non-anticipating)
in the following sense. For every $n$ and $\uu{a}^n,\tilde{\uu{a}}^{n}\in\uu{\calA}^n$,
denoting $\uu{B}^n=\beta^n(\uu{a}^n)$, $\tilde{\uu{B}}^n=\beta^n(\tilde{\uu{a}}^n)$,
\begin{equation}\label{25}
\text{if $\uu{a}^n_k=\tilde{\uu{a}}^n_k$ for all $k\leq K$
then $\uu{B}^n_t=\tilde{\uu{B}}^n_t$ for all $t\in[0,(K+1)n^{-1})$.}
\end{equation}

A strategy $\beta^n$ is said to
be {\it admissible} for the $n$-th system if \eqref{23}, \eqref{24}
and \eqref{25} hold.
The collection of all admissible strategies for the $n$-th system
is denoted by ${\mathbb B}^n$.
The relation \eqref{26}, that defines the process $B^n$ in terms of
$a^n$ ``omega-by-omega'', will be written as $B^n=\beta^n(a^n)$ in what follows.

By our definitions, the process $a^n_k$ is $\{\calF^n_k\}$-adapted.
As a consequence, the process $A^n_t$ is $\{\calG^n_t\}$-adapted.
For the adaptedness of $B^n$ we have the following.

\begin{lemma}\label{lem:B is adapted}
If $a^n\in\calA^n$ and $\beta^n\in{\mathbb B}^n$ then the process
$B^n:=\beta^n(a^n)$ is $\{\calG^n_t\}$-adapted.
\end{lemma}
The proof of this result appears in the appendix.

\paragraph{Game's value.}
With a slight abuse of notation, for a given control process $B^n$ for the SC
and a control process $a^n=(\pi^n,J^n)$ for the adversary,
we define the corresponding cost
$C^n(B^n,a^n)$ by identifying it with $C^n(B^n,J^n)$ of \eqref{27}.
The value of the game is then defined by
\[
V^n=\inf_{\beta^n\in{\mathbb B}^n}\sup_{a^n\in\calA^n}C^n(\beta^n(a^n),a^n).
\]

\begin{remark}\label{rem00}
We have constructed the game in such a way
that its whole history is measurable w.r.t.\ the job size history.
Therefore, the requirement that the decisions of both players are adapted to the history
of job sizes entails
that the players are allowed to observe the complete game history
(such as the past values of $(W^n_t, B^n_t, D^n_t, Q^n_t)$)
when making their decisions.
\end{remark}

\subsubsection{Uncertainty classes and assumptions}\label{sec:uc}

For each $\pi\in\calP$, we denote by $\xi^\pi$ and $(\sigma^\pi)^2$ the mean and variance of $\pi$, i.e.
\begin{align*}
    &\xi^\pi=\int_{[0,\infty)} x\pi(dx),
    &(\sigma^\pi)^2=\int_{[0,\infty)} (x-\xi^\pi)^2\pi(dx).
\end{align*}
It will be convenient to work with the notation $q^\pi=\frac{1}{2}(\sigma^\pi)^2$.
Further assumptions on the uncertainty classes $\PI^{\ell n}$ are as follows.
\begin{assumption}\label{ass:Pi} (Uncertainty classes).
\begin{enumerate}
    \item \label{it:4th moment} $\sup_{n,\ell}\sup_{\pi\in\PI^{\ell n}}\int_{[0,\infty)}x^4\pi(dx)<\iy$.
	\item \label{it:K,b}
	For each $\ell$ there exist a constant $\mu^\ell>0$ and a compact set
	$K^{\ell}\subset\R\times(0,\iy)$ such that
	 $K^{\ell n}\to K^\ell$ in $d_{\rm H}$, the Hausdorff distance
	corresponding to the Euclidean metric on $\R^2$, where
	\[
	K^{\ell n}=\{(b^{\pi \ell n},q^\pi):\pi\in\PI^{\ell n}\},
	\qquad
	b^{\pi\ell n}:=n^{1/2}(\xi^\pi-1/\mu^\ell).
	\]
	\end{enumerate}
\end{assumption}
Thus $K^{\ell n}$ is the collection of prelimit drift and (half)
variance of all measures that lie in $\PI^{\ell n}$.
It follows from the assumption that $\PI^{\ell n}$ are closed for each $\ell$ and $n$
and Assumption \ref{ass:Pi}(1) that $K^{\ell n}$ are closed sets.
The assumption implies
that all service time distributions allowed to be used by the adversary
have the property that their mean converges to $1/\mu^{\ell}$ as $n\to\iy$.
Thus $\mu^\ell$ represents the first order service rate for type $\ell$.
These parameters are assumed to satisfy the following.

\begin{assumption}\label{assn:ht} (Heavy traffic).
The constants $\mu^\ell$ from Assumption \ref{ass:Pi} satisfy
$\sum_\ell (\mu^\ell)^{-1}=1$.
\end{assumption}

Assumptions \ref{ass:Pi} and \ref{assn:ht}
are our standing assumptions, that will be in force throughout this paper.
Let us label the classes so that
\begin{equation}\label{31}
h^1\mu^1\leq h^2\mu^2\leq ...\leq h^L\mu^L.
\end{equation}
By scaling $h$ we may and will assume without loss that $h^1\mu^1=1$.

\subsection{Main results}\label{sec22}

A {\it non-preemptive fixed priority policy} is a strategy for the SC according to which
service is noninterruptible and whenever the server becomes available
(or the system is empty and a new arrival occurs)
it admits into service a job of the type that has lowest index waiting in the queue
(or arriving) at that moment. Because we have labeled the types as in \eqref{31},
the corresponding policy obtained prioritizes
according to the index $h\mu$. This is often referred to in the literature as the
{\it nonpreemptive $c\mu$ policy}.
We will denote by $\beta^{*n}\in \B^n$ the strategy corresponding to
the nonpreemptive $c\mu$ policy for the $n$-th system.
Our first main result is that the sequence of strategies $\beta^{*n}$
is AO regardless of the behavior of the adversary.

\begin{theorem}\label{thm:cmu is opt}
Consider an arbitrary sequence of controls for the adversary, $a^n\in\calA^n$.
Then
\[
\lim_n\{\inf_{\beta^n\in\B^n}C^n(\beta^n(a^n),a^n)-C^n(\beta^{*n}(a^n),a^n)\}=0.
\]
\end{theorem}

Theorem \ref{thm:cmu is opt} allows us to reduce the asymptotic treatment of the game to that of a control problem in which the strategy for the SC has been fixed as $\beta^{*n}$
 and one optimizes only over the adversary controls. 
We now introduce a problem that describes the limiting behavior of
this control problem, hence of the game.

\paragraph{Martingale control problem (MCP).}

Given $w\in\R_+$, an {\it admissible control system for $w$} is a tuple
\[
\calS=(\bar\Om,\bar\calF,\{\bar\calF_t\},\bar\PP,(\sfX^\ell,\sfY^\ell)_\ell,\sfW^\tot,\sfR),
\]
where $(\bar\Om,\bar\calF,\{\bar\calF_t\},\bar\PP)$ is a filtered probability space,
and
\begin{enumerate}
    \item $\sfX^\ell$ are continuous $\bar\calF_t$-martingales with $\sfX_0^\ell=0$,
    \item $\sfY^\ell$ are
    $\bar\calF_t$-adapted with $\sfY^\ell_0=0$,
    \item $\bar\PP$-a.s., $(\sfW^\tot,\sfR)=\Gam(w+\sfX^\tot+\sfY^\tot)$,
    \item \label{it:dXy in chK} One has $\bar\PP$-a.s. $[\sfX^\ell,\sfX^{\ell'}]=0$ for all $\ell\neq\ell'$ and
    \begin{align*}
        \Big(\frac{\sfY^\ell_t-\sfY^\ell_s}{t-s},\frac{[\sfX^\ell]_t-[\sfX^\ell]_s}{2(t-s)}\Big)\in\text{ch}(K^\ell),\qquad 0\le s<t.
    \end{align*}
\end{enumerate}
Notice that Property 4 implies that both $\sfY^\ell$ and $[\sfX^\ell]$ have Lipschitz sample paths.
The collection of all admissible control systems for $w$ is denoted by
$\frS_w$.
The cost for the problem is a function from $\cup_{w\geq 0}\frS_w$ to $\R_+$, given by
\[
\sfC(\calS)=\bar\E\Big[\int_0^{\infty}e^{-t}\sfW^\tot_tdt\Big],
\]
where $\bar\E$ denotes expectation under $\bar\PP$,
and the value is defined as
\[
\sfV(w)=\sup_{\calS\in\frS_w}\sfC(\calS).
\]

In the MCP, the process $\sfW^\tot$ represents the limit
of the processes $\hat W^{\tot,n}=\sum_\ell\hat W^{\ell n}$
and $\sfR$ the limit of $\hat R^n$.
The processes $\sfX^\ell$ and $\sfY^\ell$ are related
to the arrival processes $A^n$ in a slightly more complicated way, to be explained
in \S\ref{sec:ao}.
In terms of notation, $\sfX^\tot$ and $\sfY^\tot$ stand for the sum over $\ell$ as per our
convention, but note that as far as workload is concerned,
the formulation above involves only one process, $\sfW^\tot$
(a notation $\sfW^\ell$ for the limit of $\hat W^{\ell n}$ is not needed).

For $(v_1,v_2,b,q)\in\R^4$ let
\begin{align*}
&\bar{\mathbb{H}}(v_1,v_2,b,q)=bv_1+qv_2,
\\
&\mathbb{H}^\ell(v_1,v_2)=\max_{(b,q)\in\text{ch}( K^\ell)}\bar{\mathbb{H}}(v_1,v_2,b,q),
\\
&\mathbb{H}^{\ell n}(v_1,v_2)=\max_{(b,q)\in\text{ch}( K^{\ell n})}\bar{\mathbb{H}}(v_1,v_2,b,q).
\end{align*}
Then the value of the MCP can be characterized in terms of the following HJB equation
\begin{align}\label{eq:HJB}
\sum_{\ell=1}^L\mathbb{H}^\ell(u'(w),u''(w))-u(w)+w=0,\quad w\geq 0,\tag{HJB}
\end{align}
that will always be considered
with boundary conditions $u'(0)=0$ and $\limsup_{w\to\infty}|u(w)|/w<\infty$.
We will be concerned with classical solutions, namely $\calC^2$ function
that satisfy the equation classically.

\begin{proposition}\label{pr:verification}
1. There exists a unique classical solution to \eqref{eq:HJB}, denoted throughout by $u$.
Moreover, $\sfV=u$.
\\
2. Given $w$ there exists an admissible control system which is optimal for the MCP.
Under this system, $(\sfW^\tot,\sfR)$ form a weak solution $(W,R)$ to the SDE
with reflection at the origin
\[
W_t=w+\int_0^tb(W_s)ds+\int_0^t\sig(W_s)dZ_t+R_t,
\]
where $Z$ is a standard BM (SBM).
Above, $b=\sum_\ell b^\ell$ and $\sig=(\sum_\ell(\sig^\ell)^2)^{1/2}$.
Moreover, if we set $q^\ell=\frac{1}{2}(\sig^\ell)^2$ then
$(b^\ell,q^\ell)$ is a Borel measurable function
$\R_+\to {\rm ch}(K^\ell)$ which satisfies
\[
\bar{\mathbb{H}}(u'(z),u''(z),b^\ell(z),q^\ell(z))=
\mathbb{H}^\ell(u'(z),u''(z)),
\qquad z\in\R_+.
\]
3. The function $u$ is nonnegative, non-decreasing and convex.
\end{proposition}

Our second main result relates the stochastic game to the MCP and HJB equation.
\begin{theorem}\label{thm:Vn to V}
1. One has
$V^n\to\sfV(0)=u(0)$.
\\
2. There exists for each $n$ and $\ell$ a measurable function
$\psi^{\ell n}:\R_+\to\PI^{\ell n}$  such that
\[
\bar{\mathbb{H}}(u'(z),u''(z),b^{\psi^{\ell n}(z),\ell n},q^{\psi^{\ell n}(z)})
=
\mathbb{H}^{\ell n}(u'(z),u''(z)),\qquad z\in\R_+.
\]
3. If the SC uses the strategy $\beta^{*n}$
then it is AO for the adversary to select the control according to
\begin{equation}\label{301}
\pi^{\ell n}_k=\psi^{\ell n}(\hat W_{(k-1)/n}^{\tot,n}),\qquad k\ge1.
\end{equation}
Moreover, if we let
\[
v_k^{\ell n}=(u'(\hat W^{\tot,n}_{(k-1)/n}),u''(\hat W^{\tot,n}_{(k-1)/n})),
\qquad k\ge1,
\]
then the selection expressed by \eqref{301} can equivalently be stated
as letting
\[
(b,q)\in\argmax_{(b',q')\in K^{\ell n}}(b',q')\cdot v_k^{\ell n}
\]
and then selecting a member $\pi^{\ell n}_k\in\PI^{\ell n}_k$ such that
its prelimit drift and variance coefficients are given by
$(b^{\pi \ell n},q^{\pi \ell n})=(b,q)$.
\end{theorem}

\section{Comments and examples}
\label{sec3}

\subsection{On the drift-variance tradeoff}
\label{sec31}

Let us describe the simplest setting of our model for which
the MCP is not trivial.
This is the case when there is only one type of jobs,
hence the SC has no freedom at all and the server simply works
whenever there is work in the system. Moreover,
the uncertainty class consists of only two members
for each $n$. The adversary dynamically selects one of the two members
for each arrival.
This version of the model, although very simple,
is meaningful, and can be used to explain
the so called {\it drift-variance tradeoff}, which is present in a more
complex fashion in the full model.

Consider a diffusion process $W$ on $\R_+$ with reflection at zero,
for which the drift and diffusion coefficients can be dynamically controlled
as follows.
Two pairs of real numbers $(b_m,\sig_m)\in\R\times(0,\iy)$,
$m=1,2$, are given. A mode $m\in\{1,2\}$ can be dynamically chosen,
and accordingly the instantaneous drift-diffusion pair of $W$ is given by
$(b_m,\sig_m)$. One attempts to maximize the cost $E\int_0^\iy e^{-t}W_tdt$.
It is not hard to see that the MCP is equivalent to this problem
in the special case under consideration.
Now, if only the drift coefficient is subject to control
(that is, $\sig_1=\sig_2$) then it can be shown by a simple coupling that
a control that always selects the larger value will maximize the cost. Similarly,
if only the diffusion coefficient can be controlled, it is the one with
larger value that is optimal to select. The situation is different when
$b_1>b_2$ and $\sig_1<\sig_2$. This problem was solved
in \cite{Sheng78}, where it was called the {\it tortoise-hare problem},
and some extensions of it were studied in \cite{ACR} where it was referred to
as the {\it drift-variance tradeoff}. As shown in \cite{Sheng78}, the HJB equation
can be fully solved, and its solution reveals an optimal
tradeoff between the modes. Specifically, under the optimal control,
mode $1$ is selected at times when $W_t\ge w^*$ and mode $2$ when $W_t<w^*$.
Here, $w^*$ is a free boundary point characterized by
an equation expressing the so called {\it principle of smooth fit}
(see \cite{Sheng78} and \cite{ACR} for more details).

Thus, in the case where the uncertainty class has two members,
the asymptotics of our model can be
described by an explicit formula (i.e., the solution from \cite{Sheng78}
to the HJB equation), and moreover the asymptotic
behavior of the adversary is fully understood.

We will have more to say about finite uncertainty classes.
In any case, in its full generality, the MCP we study in this paper
can be viewed as the problem of finding an optimal drift-variance tradeoff
albeit in a more complex setting.

\subsection{On the extremal and dominating subsets of ${\rm ch}(K^{\ell n})$}
\label{sec32}


Our results show how the adversary behaves under an AO play.
In particular, Theorem \ref{thm:Vn to V} asserts that
the prelimit drift and variance coefficients are chosen
dynamically so as to maximize the expression
$\bar{\mathbb{H}}(v_1,v_2,b,q)$ over $(b,q)\in{\rm ch}(K^{\ell  n})$
for suitably defined $(v_1,v_2)$ depending on the current state.
A further insight on how the adversary acts can be gained
by arguing that the maximization may be restricted to a smaller
set.
Because $\bar{\mathbb{H}}$ is affine in $(b,q)$, it is clear that
the maximum is in the set, denoted $\pl_{\rm ext} K^{\ell n}$,
of extreme points of ${\rm ch}(K^{\ell  n})$.
We argue that one may restrict further the set where the maximum occurs.
Consider the partial order $\le$ on $\R^2$ defined by
$(b,q)\le(b',q')$ if $b\le b'$ and $q\le q'$.
For a compact set $K\subset\R^2$ let $\pl_{\rm dom}K$
denote the set of dominating points in $K$ w.r.t.\ this partial order.
That is, the smallest set of points $(b',q')$ such that
for  every $(b,q)\in K$ there exists $(b',q')$ in this set such that
$(b,q)\le(b',q')$.
In follows from Proposition \ref{pr:verification} that $u'$ and $u''$
are non-negative functions. As a result, $\bar{\mathbb{H}} (u'(w),u''(w),b,q)$
is increasing in $b$ and $q$. This implies that it is always advantageous
for the adversary to choose points in the set $\pl_{\rm dom}K^{\ell n}$.
Combining the two observations, the maximum can be restricted a priori to
\[
\pl_{\rm ext}K^{\ell n}\cap\pl_{\rm dom}K^{\ell n}.
\]
An example is shown in Figure \ref{fig:difpoints}.

The simplification becomes much more significant in the special case of
polygons. In this case the sets
$M^{\ell n}:=\pl_{\rm ext}K^{\ell n}\cap\pl_{\rm dom}K^{\ell n}$
and $M^{\ell}:=\pl_{\rm ext}K^{\ell}\cap\pl_{\rm dom}K^{\ell}$
are finite. The HJB equation simplifies to
\begin{align*}
	\sum_{\ell=1}^L\max_{(b,q)\in M^{\ell}}\bar{\mathbb{H}}
	(u'(w),u''(w),b,q)-u(w)+w=0,\qquad w\geq 0.
\end{align*}
We arrive at a conclusion that uncertainty classes given
by polygonal domains are no different than finite uncertainty classes
as far as our setting is concerned.
Beyond insight on the behavior of the adversary, this observation
is relevant also for numerical solutions to the HJB equation,
where it becomes much more manageable to numerically solve the
equation when optimizing over a finite set.
\begin{center}
	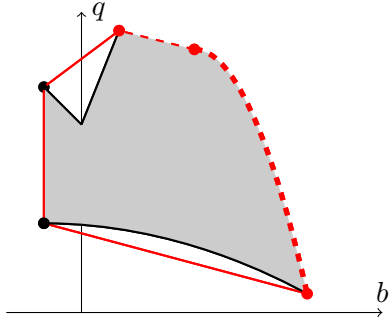
\begin{figure*}
	\begin{center}
		\begin{tikzpicture}[scale=1]
		\draw[->] (-1, 0) -- (4, 0) node[above] {$b$};
		\draw[->] (0, 0) -- (0,4) node[right] {$q$};
		\filldraw[black] (-0.5,3) circle (2pt) node[anchor=south east]{};
		\filldraw[red] (0.5,3.75) circle (2pt) node[anchor=south east]{};
		\filldraw[red] (1.5,3.5) circle (2pt) node[anchor=south east]{};
		\filldraw[black] (-0.5,1+0.25*0.75) circle (2pt) node[anchor=south east]{};
		\filldraw[red] (3,0.25) circle (1pt) node[anchor=south east]{};
		\draw[red,  thick, dashed, name path= C] (1.5,3.5) -- (0.5,3.75);
		\draw[black,  thick, name path= D] (0,2.5) -- (0.5,3.75);
		\draw[black,  thick, name path= E] (-0.5,3) -- (0,2.5);
		\draw[red, thick, name path= F] (-0.5,3) -- (-0.5,1+0.25*0.75);
		\draw[red, line width=2pt, dashed, name path= B]  (1.5,3.5) parabola                 (3,0.25);
		\draw[red,  thick, name path= G]  (-0.5,1+0.25*0.75) --                  (3,0.25);
		\draw[red,  thick, name path= H]  (-0.5,3) --                   (0.5,3.75);
		\draw[black,  thick, name path=A]   (-0.5,1+0.25*0.75) parabola               (3,0.25);
		\tikzfillbetween[of=A and B]{gray!40}
		\tikzfillbetween[of=A and C]{gray!40}
		\tikzfillbetween[of=A and E]{gray!40}
		\draw[red,  thick, dashed] (1.5,3.5) -- (0.5,3.75);
		\draw[black,  thick] (0,2.5) -- (0.5,3.75);
		\draw[black,  thick] (-0.5,3) -- (0,2.5);
		\draw[red, thick] (-0.5,3) -- (-0.5,1+0.25*0.75);
		\draw[red, line width=2pt, dashed]  (1.5,3.5) parabola                 (3,0.25);
		\draw[red,  thick]  (-0.5,1+0.25*0.75) --                  (3,0.25);
		\draw[red,  thick]  (-0.5,3) --                   (0.5,3.75);
		\draw[black,  thick]   (-0.5,1+0.25*0.75) parabola               (3,0.25);
		\filldraw[black] (-0.5,3) circle (1pt) node[anchor=south east]{};
	\filldraw[red] (0.5,3.75) circle (2pt) node[anchor=south east]{};
	\filldraw[red] (1.5,3.5) circle (2pt) node[anchor=south east]{};
	\filldraw[black] (-0.5,1+0.25*0.75) circle (2pt) node[anchor=south east]{};
	\filldraw[red] (3,0.25) circle (2pt) node[anchor=south east]{};
		\end{tikzpicture}
	\end{center}
		\caption{\sl An example of a set $K$ and its extremal and dominating sets.
The set $K$ is shown in gray.
The boundary of ${\rm ch}(K)$ is in red.
The set $\pl_{\rm ext} K$ consists of the two black points, three red
points and the red dashed bold curve.
The set $\pl_{\rm dom} K$ consists of the three red points, the red dashed 
and the red bold dashed curve. Thus the set $\pl_{\rm ext} K\cap\pl_{\rm dom} K$
is given by the three red points and the red bold dashed curve.
\label{fig:difpoints}}
	\end{figure*}
\end{center}

\subsection{Example: Gamma distribution}

Gamma distributions are sometimes used to model service duration
in the context of queueing systems; see for example
\cite{tak61} for an early such reference.
Our goal here is to determine what collections of Gamma distributions
give rise to a reasonable choice of
an uncertainty class $K^{\ell n}$, and also what $K^{\ell n}$
may look like for a natural choice of an uncertainty class
defined in terms of the distribution parameters.
Because there is no difference in our treatment between the various
types $\ell\in[L]$, fix $\ell$ and let it be removed from all notation.

For $\al>0$, $\beta>0$,
Gamma$(\al,\beta)$ is the distribution on $\R_+$ that has a density
$f$ given by
\[
f(x)=\frac{\beta^\al}{\Gam(\al)}x^{\al-1}e^{-\beta x}.
\]
The mean and variance are given by $\al/\beta$ and $\al/\beta^2$,
respectively.
For the prelimit  coefficients, consider
a target set $K^n$ given as a rectangle,
$$
K^{n}=[b^n_1,b^n_2]\times[q^n_1,q^n_2],
$$
and assume that $b^n_1$, $b^n_2$, $q^n_1$ and $q^n_2$
converge to $b_1,b_2, q_1$ and  $q_2$, respectively.
Consider $\PI^n$ to be a collection of
Gamma$(\alpha,\beta)$ distributions of the form
\begin{equation}\label{300}
\PI^n=\{\text{Gamma}(\al,\beta):(\al,\beta)\in G^n\},
\end{equation}
where $G^n\subset(0,\iy)^2$.
What should $G^n$ be in order to achieve the above target set?
To answer this, compute the prelimit coefficients for
$\pi\in \PI^n$ with parameters $\alpha,\beta$. Then
\[
b^{\pi}=\sqrt{n}\Big(\dfrac{\alpha}{\beta}-\mu\Big) \quad\text{ and }\quad q^\pi=\dfrac{\alpha}{2\beta^2}.
\]
Note first that $b\in [b^n_1,b^n_2]$ implies that 
$$\dfrac{\alpha}{\beta}\in \mu+\Big[\dfrac{b^n_1}{\sqrt{n}},\dfrac{b^n_2}{\sqrt{n}}\Big].$$
In addition,
$$q^\pi=\dfrac{\alpha}{2\beta^2}=\dfrac{1}{2\beta}\times\dfrac{\alpha}{\beta}\in[q^n_1,q^n_2] .$$
The last two equations imply
$$\beta\in \bigg[\dfrac{1}{2q^n_2(\mu+\frac{b^n_2}{\sqrt{n}})}, \dfrac{1}{2q^n_1(\mu+\frac{b^n_1}{\sqrt{n}})}\bigg],$$
and 
$$\alpha\in \mu\beta+\beta \Big[\dfrac{b^n_1}{\sqrt{n}},\dfrac{b^n_2}{\sqrt{n}}\Big].$$
This gives rise to
$$G^n=\left\lbrace (\alpha,\beta),\, \beta\in \bigg[\dfrac{1}{2q^n_2(\mu+\frac{b^2_2}{\sqrt{n}})}, \dfrac{1}{2q^n_1(\mu+\frac{b^1_2}{\sqrt{n}})}\bigg],\, \alpha\in \mu\beta+\beta \Big[\dfrac{b^n_1}{\sqrt{n}},\dfrac{b^n_2}{\sqrt{n}}\Big]\right\rbrace.$$

 
The reverse direction is also of interest, namely for a given "reasonable"
choice of $G^n$ we shall ask what $K^n$ and $K$ look like.
Consider first $\PI^n$ given as in \eqref{300}, where
$G^n=[\alpha^n_1,\alpha^n_2]\times[\beta^n_1,\beta^n_2]$.
We argue that in this case the disturbances must be $o(1)$
to meet our assumptions,
hence this choice does not lead to what we have called
a high uncertainty regime. Indeed, this choice leads to
\[b^\pi\in\Big[\sqrt{n}\Big(\dfrac{\alpha^n_1}{\beta^n_2}-1\Big), \sqrt{n}\Big(\dfrac{\alpha^n_2}{\beta^n_1}-1\Big)\Big]\text{ and } q^\pi\in \Big[\dfrac{\alpha_1^n}{2(\beta^n_2)^2},\dfrac{\alpha_2^n}{2(\beta^n_1)^2}\Big].
\]
In order for $K^{n}$ to converge to a compact set
it is necessary (but not sufficient) that $\sqrt{n}(\alpha^n_i/\beta^n_k-\mu)=O(1)$
for all $i,k\in \lbrace 1,2 \rbrace$. 
Hence one must have $\lvert \beta^n_1-\beta^n_2\rvert =O(n^{-1/2})$.
For $\pi$ such that $b^\pi=O(1)$, 
\begin{equation}
\label{eq:gammamean}\dfrac{\alpha}{\beta}=\mu+O(\tfrac{1}{\sqrt{n}}),
\end{equation}
which means 
\begin{equation}\label{eq:gammavariance}q^\pi=\dfrac{\alpha}{2\beta^2}
=\dfrac{\mu}{2\beta}+O(\tfrac{1}{\sqrt{n}}).\end{equation}
For the set $K^n$ to converge, $\beta^n_1$ and $\beta^n_2$ must
converge, and because of \eqref{eq:gammavariance}
and $\lvert \beta^n_1-\beta^n_2\rvert \to 0$ they must have
the same limit $\beta^*$.
There is only one choice of limiting variance given by $\mu/(2\beta^*)$.
Although our results are not vacuous in this case, they do not
give significant information.

We thus turn to a different definition of $G^n$, which makes
$\al/\beta$ nearly constant. A natural set that satisfies this
requirement is
\[
G^n=\{(\al,\beta):\beta\in[\beta_1,\beta_2],
\al-\beta\in[-\al^n_1,\al^n_2]\},
\]
for $\sqrt{n}\al^n_i\to\al_i\in\R$.
We can then introduce a new parameterization $(b,\beta)$,
where $b=\sqrt{n}(\al/\beta-\mu)$.
If $\pi=\text{Gamma}(\alpha,\beta)$, we obtain
$b^\pi=b$ and $q^\pi=\mu/(2\beta)+b/(2\beta\sqrt{n})$.
This gives rise to
$$
K^n=\left\lbrace \Big(b,\dfrac{\mu}{2\beta}+\dfrac{b}{2\beta\sqrt{n}}\Big),
\, \beta\in [\beta_1,\beta_2],\,
b\in\Big[-\sqrt{n}\dfrac{\alpha^n_1}{\beta}, \sqrt{n}\dfrac{\alpha^n_2}{\beta}\Big]\right\rbrace.
$$
Taking the Hausdorff limit, using \eqref{eq:gammavariance}, $K^n$ converges to 
\begin{equation}\label{310}
K=\left\lbrace (b,q),  q\in \Big[\dfrac{\mu}{2\beta_2},\dfrac{\mu}{2\beta_1}\Big],\, b\in \Big[-\dfrac{2\al_1}{q},\dfrac{2\al_2}{q}\Big]\right\rbrace.
\end{equation}
Possible sets $K$ corresponding to the above are shown in Figure \ref{fig1}.

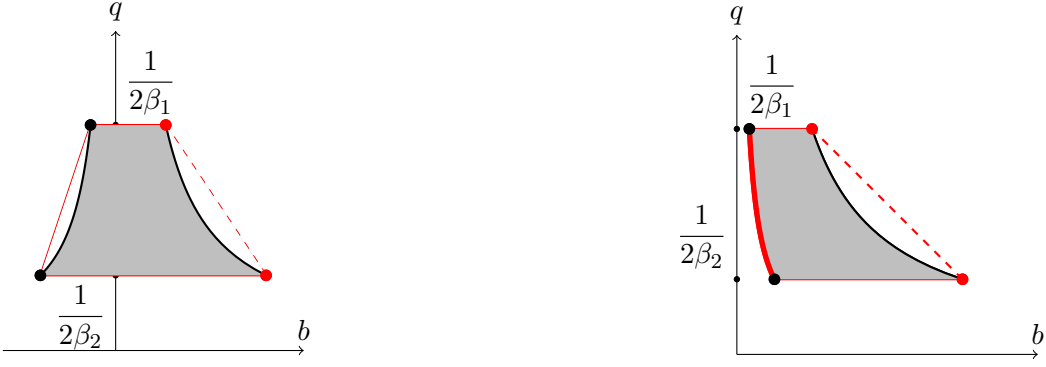
\begin{figure*}[t]
	\begin{subfigure}[b!]{0.49\textwidth}
	\begin{tikzpicture}[scale=1, rotate=90, yscale=-1]
	\draw[->, name path=C] (0, 0) -- (4.25, 0) node[above] {$q$};
	\draw[->, name path=D] (0, -1.5) -- (0,2.5) node[above] {$b$};
	\draw[ domain=1:3, smooth, thick, variable=\x, black,  name path=A] plot ({\x}, {-1/\x});
	\draw[ domain=1:3, variable=\x, black, thick , name path=B]  plot ({\x}, {2/\x});
	\filldraw[black] (1,0) circle (1pt) node[anchor=north east]{$\dfrac{1}{2\beta_2}$};
	\filldraw[black] (3,0) circle (1pt) node[anchor=south west]{$\dfrac{1}{2\beta_1}$};
	\draw[ domain=-1:2, smooth, variable=\y, red,  thick, name path=E] plot ({1}, {\y});
	\draw[ domain=-1/3:2/3, smooth, variable=\y, red,  thick, name path=F] plot ({3}, {\y});
	\draw[red,  thick, dashed, name path=G] (1,2) -- (3,2/3);
	\draw[red,  thick, name path=H] (1,-1) -- (3,-1/3);
	\filldraw[red] (1,2) circle (2pt) node[anchor=south east]{};
	\filldraw[red] (3,2/3) circle (2pt) node[anchor=south east]{};
	\filldraw[black] (1,-1) circle (2pt) node[anchor=south east]{};
	\filldraw[black] (3,-1/3) circle (2pt) node[anchor=south east]{};
	\tikzfillbetween[of=E and F]{gray!50}
		\tikzfillbetween[of=A and H]{white}
				\tikzfillbetween[of=B and G]{white}
		\draw[ domain=1:3, smooth, thick, variable=\x, black] plot ({\x}, {-1/\x});
	\draw[ domain=1:3, variable=\x, black, thick ]  plot ({\x}, {2/\x});
	\filldraw[red] (1,2) circle (2pt) node[anchor=south east]{};
\filldraw[red] (3,2/3) circle (2pt) node[anchor=south east]{};
\filldraw[black] (1,-1) circle (2pt) node[anchor=south east]{};
\filldraw[black] (3,-1/3) circle (2pt) node[anchor=south east]{};
	\end{tikzpicture}
		\end{subfigure}
	\begin{subfigure}[b!]{0.49\textwidth}
	\begin{tikzpicture}[scale=1, rotate=90, yscale=-1]
	\draw[->] (0, 0) -- (4.25, 0) node[above] {$q$};
	\draw[->] (0, 0) -- (0,4) node[above] {$b$};
	\draw[ domain=1:3, smooth, variable=\x, red, line width=2pt, name path=A] plot ({\x}, {0.5/\x});
	\draw[ domain=1:3, variable=\x, black, thick, name path= B]  plot ({\x}, {3/\x});
	\filldraw[black] (1,0) circle (1pt) node[anchor=south east]{$\dfrac{1}{2\beta_2}$};
	\filldraw[black] (3,0) circle (1pt) node[anchor=south west]{$\dfrac{1}{2\beta_1}$};
	\draw[ domain=0.5:3, smooth, variable=\y, red,  thick, name path=C] plot ({1}, {\y});
	\draw[ domain=1/6:1, smooth, variable=\y, red,  thick, name path=D] plot ({3}, {\y});
	\draw[red,thick, dashed] (1,3) -- (3,1);
	\filldraw[red] (1,3) circle (2pt) node[anchor=south east]{};
	\filldraw[red] (3,1) circle (2pt) node[anchor=south east]{};
	\filldraw[black] (1,0.5) circle (2pt) node[anchor=south east]{};
	\filldraw[black] (3,1/6) circle (2pt) node[anchor=south east]{};
	\tikzfillbetween[of=A and B]{gray!50}
		\draw[ domain=1:3, smooth, variable=\x, red, line width=2pt] plot ({\x}, {0.5/\x});
	\draw[ domain=1:3, variable=\x, black, thick]  plot ({\x}, {3/\x});
		\filldraw[red] (1,3) circle (2pt) node[anchor=south east]{};
	\filldraw[red] (3,1) circle (2pt) node[anchor=south east]{};
	\filldraw[black] (1,0.5) circle (2pt) node[anchor=south east]{};
	\filldraw[black] (3,1/6) circle (2pt) node[anchor=south east]{};
	\end{tikzpicture}
			\end{subfigure}
\caption{\label{fig1}
	\sl
The set $K$ of \eqref{310}. The extreme and dominating sets are shown according
to the same convention as in Figure \ref{fig:difpoints}.
Left: $\alpha_1>0$ and $\al_2>0$. Right: $\al_1<0$ and $\al_2>0$. (The case $\al_1<0$ and $\al_2<0$ is similar to the first one). In both cases, the set of extreme dominating
points is given by the two red points.
}
\end{figure*}


\subsection{Example: Finite uncertainty class}

Assume that for each $\ell$, the uncertainty class $\PI^{\ell n}$
consists of finitely many measures. For simplicity, assume also that
they all have the same cardinality, $M$. We label the members of
each uncertainty class by $m\in[M]$, and, as in
\S \ref{sec31}, refer to $m$ as modes.
Thus each $\PI^{\ell n}$ consists of $\{\pi^{\ell mn}:m\in[M]\}$.
In this case the game can be described in very simple terms.
In the $n$-th system,
for each $(\ell,m)$ there is an IID sequence of $\pi^{\ell m n}$-distributed
RVs, denoted $J^{\ell m n}_k$, $k\in\N$. These sequences are
mutually independent. At each arrival time, $k/n$, for each type $\ell$,
there are $M$ job candidates, namely $J^{\ell m n}_k$, $m\in[M]$.
The adversary selects one of the $M$ candidates and this will
be the job to actually arrive at buffer $\ell$. When making the selection,
the adversary has access to the history of the system,
as well as to the label $m$ (thus the distribution)
of each of the candidates, but it does not
have access to the RVs $J^{\ell m n}_k$ themselves.
As in our general setting, the SC's role is to split the server's
effort among the types.
See Figure \ref{fig:Finite class}.

\begin{figure*}
	\centering
	\begin{tikzpicture}[scale=1.5]
		\coordinate (l1ul) at (1.5,2.5);
		\coordinate (lLul) at (5,2.5);
		\coordinate (l1dl) at (1.5,2);
		\coordinate (lLdl) at (5,2);
		\coordinate (l1ur) at (2,2.5);
		\coordinate (lLur) at (5.5,2.5);
		\coordinate (l1dr) at (2,2);
		\coordinate (lLdr) at (5.5,2);
		\node (DM)   at (3.5,1) {};
		\coordinate (s)   at (3.5,0.25);
		
		\draw[thick] (l1ul)--(l1dl)--(l1dr)--(l1ur);
		\draw[thick] (lLul)--(lLdl)--(lLdr)--(lLur);
		\node[thick] at (3.5,2.25) {$\dots$};
		\draw[thick] (s) circle (0.25); 
		\node[anchor=west] at (3.75,0.25) {Server};
		\node (AAA) at ( $(l1dr)+(1.2,-0.7)$ ) {};
		\draw[thick,-o] (l1dr) -- ( $(l1dr)+(1.2,-0.7)$ );
		\draw[thick,-o] (lLdl)-- ++(-1.2,-0.7);
		\draw[thick,o->] (3.5,1)-- ++(0,-0.45); 
		\draw[thick,o-] (1.55,2.85)-- ++(-0,0.35);
		\draw[thick,o-] (1.95,2.85)-- ++(0,0.35);
		\draw[thick,o-] (5.05,2.85)-- ++(-0,0.35);
		\draw[thick,o-] (5.45,2.85)-- ++(0,0.35);
		\draw[thick,dotted] (0.5,2.75) -- (6,2.75);
		\draw[thick,o->] (1.75,2.625) -- (1.75,2.3);
		\draw[thick,o->] (5.25,2.625) -- (5.25,2.3);
		\node[anchor=west] at (6,2.75) {Adversary};
		\draw[thick] (5.05,2.85)--(5.25,2.625);
		\draw[thick] (1.55,2.85)--(1.75,2.625);
		\draw[thick] ( $(l1dr)+(1.2,-0.7)$ )--(3.5,1);
		\draw[thick,dotted] (0.5,1.1) -- (6,1.1);
		\node[anchor=west] at (6,1.1) {SC};
	\end{tikzpicture}
	\caption{\label{fig:Finite class}
	\sl
Illustration of the case with finite uncertainty classes, specifically
$M=2$. With two streams of IID job sizes per type, the adversary dynamically controls
which stream is admitted into the buffer, and the SC controls the server's
effort allocation.}
\end{figure*}
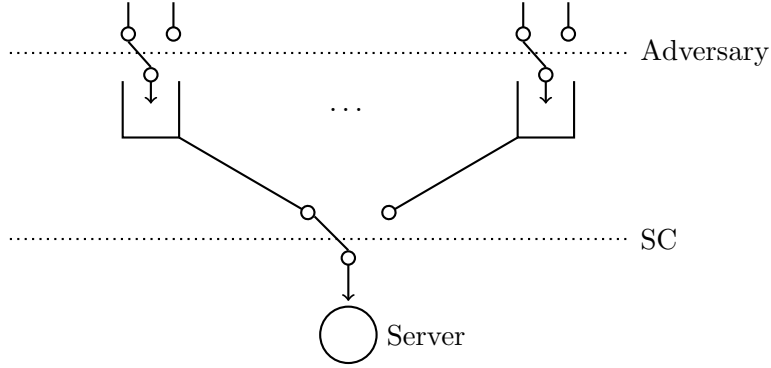

It is assumed that all distributions alluded to above have
a finite fourth moment, so as to satisfy our assumptions.
Next, to meet our scaling assumptions, it is assumed that
\begin{align*}
	&\E[J_k^{\ell mn}]=\frac{1}{\mu^\ell}+\frac{b^{\ell m}}{\sqrt{n}}+o(\tfrac{1}{\sqrt{n}}),
	&	\frac{1}{2}\text{Var}(J_k^{\ell mn})=q^{\ell m}+o(1),
\end{align*}
where $q^{\ell m}>0$.
By definition, for $\pi=\pi^{\ell mn}$, we have
$b^{\pi \ell n}=b^{\ell m}+o(1)$.
As a consequence, $K^{\ell,n}\to K^\ell$ in $d_H$, where $K^\ell$
consists of the pairs $(b^{\ell m},q^{\ell m})$.


We now describe the choices made by the adversary.
As argued in \S\ref{sec32},
among all points in $K^{\ell n}$, only those in
$\Del^{\ell n}:=\pl_{\rm ext}K^{\ell n}\cap\pl_{\rm dom}^{\ell n}K^{\ell n}$ 
are taken into account. Let $M^{\ell n}\subset[M]$
denote the set of labels corresponding to points in $\Del^{\ell n}$.
An example of a finite set $K^{\ell n}$ is shown in
Figure \ref{fig:finite K}(a). In this example there are three member of the set
$\pl_{\rm ext}K^{\ell n}\cap\pl_{\rm dom}K^{\ell n}$, shown in red.
We argue that these three points define a finite partition of $\R_+^2$
according to which the adversary's choices are made.
To this end, recall the notation
\[
v_k^{\ell n}=(u'(\hat W^{\tot,n}_{(k-1)/n}),u''(\hat W^{\tot,n}_{(k-1)/n})).
\]
By Proposition \ref{pr:verification}, this process takes values
in $\R_+^2$.
By Theorem \ref{thm:Vn to V}, the choice made
at time $k/n$ is to select mode
\begin{align}\label{837}
m_k^{\ell n}\in\argmax_{m\in M^{\ell n}}(b_m,q_m)\cdot v_k^{\ell n}.
\end{align}
Alternatively, this can be expressed as
\[
m^{\ell n}_k\in\argmax_{m\in M^{\ell n}}\|(b_m,q_m)\|\cos(\theta_m),
\]
where $\theta_m$ is the angle between $(b_m,q_m)$
and $v^{\ell n}_k$.
This rule corresponds to
a partition of $\R_+^2$ into $|M^{\ell n}|$ cone-shaped regions,
as shown in Figure \ref{fig:finite K}(b).
Namely, the adversary chooses a mode
at the $k$-th step according to the region to which $v^{\ell n}_k$
belongs.

\begin{figure}	
	\begin{subfigure}{0.5\textwidth}
		\centering
		\hfill\includegraphics[width=8cm]{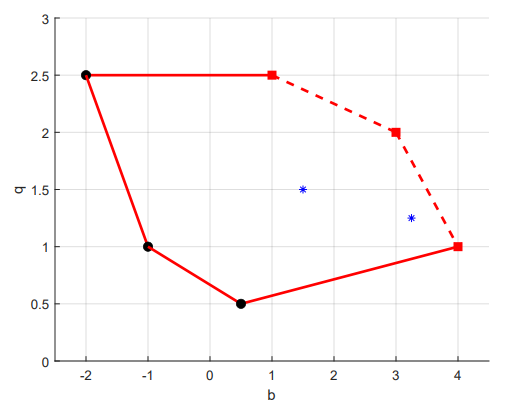} 
		\caption{}\label{subfig:dom}
	\end{subfigure}
	\begin{subfigure}{0.5\textwidth}
		\hfill\includegraphics[width=8cm]{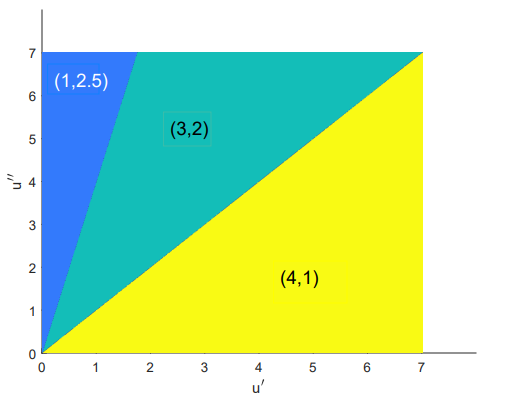}
		\caption{}\label{subfig:DecReg}
	\end{subfigure}
	\caption{
	\sl
	(a) A finite uncertainty class $K^{\ell n}$ consisting of 8 points, shown in
blue (2 points), black (3 points) and red (3 points).
The extreme points are the 3 black and 3 red points. The boundary of ${\rm ch}(K^{\ell n})$ is shown in red; its subset of dominating set is shown
in dashed red line. The intersection $\pl_{\rm ext}K^{\ell n}\cap\pl_{\rm dom}K^{\ell n}$
consists of the three red points.
	(b)
	A partition of $\R_+^2$ into decision regions.
	The decision is made according to the region to which
	$v^{\ell n}_k$ belongs.
}
\label{fig:finite K}
\end{figure}

Note that if $L=1$ and the set of points that are
extremal and dominating consists of exactly two points,
then we are back in the scenario described in \S\ref{sec31}.
The above discussion shows that in this case the decision
is according to whether
%
\begin{align*}
	(b_1,q_1)\cdot v^{\ell n}_k \lessgtr (b_2,q_2)\cdot v^{\ell n}_k.
\end{align*}

\section{AO of the $c\mu$ policy}
\label{sec:cmu}
In this section we prove Theorem \ref{thm:cmu is opt}.
The collection of all processes comprising the queueing model are uniquely defined once
the players have selected their actions.
Given $n$, a control $a^n\in\calA^n$ chosen by the adversary and
a strategy $\beta^n\in\B^n$ chosen by the SC, we shall always assume
the collection of processes is the one determined by the pair $(a^n,\beta^n)$
and will not explicitly specify the dependence on this pair.
In \S\ref{sec:uc} we defined the tuple $(\xi^\pi,b^{\pi\ell n},\sig^\pi,q^\pi)$ corresponding
to a measure $\pi$. When one substitutes the measures $\pi_k^{\ell n}$
chosen by the adversary, this tuple becomes a stochastic process,
for which we shall use the notation
\begin{align}\label{33}
    &\bxi_k^{\ell n}=\xi^{\pi_k^{\ell n}},
    &&\bb^{\ell n}_k=b^{\pi_k^{\ell n},\ell n},
    &\bsigma_k^{\ell n}=\sig^{\pi_k^{\ell n}},
    &&\bq^{\ell n}_k=q^{\pi_k^{\ell n}}.
\end{align}

We need the notion of convergence in probability uniformly
in the strategy $\beta^n$, and related notions.

\begin{definition}
Let $\{\psi^n_{a^n,\beta^n}\}$ and $\{\varphi^n_{a^n,\beta^n}(\cdot)\}$
be a sequence of RVs and, respectively, stochastic processes,
that depend on $n\in\N$, $a^n\in\calA^n$ and $\beta^n\in\B^n$.
\begin{itemize}
     \item Given a sequence $a^n\in\calA^n$,
     the sequence $\{\psi^n_{a^n,\beta^n}\}$
     is said to converge to 0 in probability uniformly in $\beta^n\in\mathbb{B}^n$ if for any $\eps>0$,
     \begin{align*}
         \sup_{\beta^n\in\mathbb{B}^n}\prob{\abs{\psi^n_{a^n,\beta^n}}\geq\eps}\to 0
     \end{align*}
 	as $n\to\infty$.
     Henceforth, this convergence will simply be called \emph{uniform in probability}.
     \item Given $a^n\in\calA^n$, the sequence $\{\varphi^n_{a^n,\beta^n}(\cdot)\}$ is said to be uniformly (in $\beta^n\in\mathbb{B}^n$) $\calC$-tight if for any
$t_0>0$ and $\eps>0$,
     \begin{align*}
         &\lim_{x\to\infty}\limsup_{n\to\infty}\sup_{\beta^n\in\mathbb{B}^n}\PP\big(\norm{\varphi^n_{a^n,\beta^n}}_{t_0}>x\big)=0,\\
         &\lim_{\delta\to 0}\limsup_{n\to\infty}\sup_{\beta^n\in\mathbb{B}^n}\prob{w_{t_0}\brac{\varphi^n_{\beta^n},\delta}>\eps}=0.
     \end{align*}
\end{itemize}
\end{definition}

We now state several lemmas and provide the proof of
Theorem \ref{thm:cmu is opt} based on them. In \S\ref{sec:lems},
these lemmas are proved.

\subsection{Proof of Theorem \ref{thm:cmu is opt}}
\label{sec41}

The first states that the total workload process does not depend on the
strategy.

\begin{lemma}\label{lem:W,beta ind}
For any $n\in\N$, $a^n\in\calA^n$, $\beta^n,\tilde\beta^n\in\B^n$, if $W^{\tot,n}$ and $\tilde W^{\tot,n}$
are the processes corresponding to $(a^n,\beta^n)$ and $(a^n,\tilde\beta^n)$, respectively, then $W^{\tot,n}=\tilde W^{\tot,n}$.
\end{lemma}

\begin{proof}
Immediate from \eqref{eq:WR=Gamma} and the definition of $A^{\tot,n}$.
\end{proof}

Next is a polynomial in time estimate on the second moment of the
workload and queue length.
\begin{lemma}\label{lem:W,Q UI}
There exists a constant $c<\infty$ such that for all $n\in\N$, $t\ge0$, $a^n\in\calA^n$ and $\beta^n\in\B^n$,
\begin{align*}
\E\big[\|\hat{W}^{\tot,n}\|_t^2]\leq c(t+1)^2,
    &&\E\big[\|\hat{Q}^{\tot,n}\|_t^2\big]\leq c(t+1)^2.
\end{align*}
\end{lemma}


The following result is a version of the RSP, by which the queue length and
workload processes are asymptotically proportional to one another.

\begin{lemma}\label{lem:Q-W to 0}
Given any $a^n\in\calA^n$,
for all $t_0>0$ and $\ell\in[L]$,
$\|\hat{Q}^{\ell n}-\mu^\ell\hat{W}^{\ell n}\|_{t_0}\to0$ uniformly in probability as $n\to\infty$.
\end{lemma}

Whereas the above three lemmas are concerned with general
behavior of the SC, our next result is about a certain property
of the specific strategy $\beta^{*n}$.
It states that when the server operates under the fixed nonpreemptive priority rule, the workload is asymptotically concentrated on the lowest-priority type.

\begin{lemma}\label{lem:W To 0}
Assume that the SC applies the strategy $\beta^{*n}$.
Then for any $a^n\in\calA^n$, $\sum_{\ell=2}^L\hat{W}^{\ell n}\to0$ in probability as $n\to\infty$.
\end{lemma}

We are now ready to prove Theorem \ref{thm:cmu is opt}.

\begin{proof}[Proof of Theorem \ref{thm:cmu is opt}]
It suffices to prove that $\liminf_n\Del^n\ge0$ where
\[
\Del^n=\inf_{\beta\in\B^n}C^n(\beta^n(a^n),a^n)-C^n(\beta^{*n}(a^n),a^n).
\]
To this end, we observe that for any strategy $\beta^n\in\mathbb{B}^n$ and $t_0>0$, 
\begin{align}\label{eq:C=sum1}
    C^n(\beta^n(a^n),a^n)&=\E\Big[\int_0^\infty e^{-t}h\cdot\hat{Q}^n_tdt\Big]\nonumber\\
    &=\E\Big[\int_0^{t_0} e^{-t}\sum_{\ell=1}^Lh^\ell(\hat{Q}^{\ell n}_t-\mu^\ell\hat{W}^{\ell n}_t+\mu^\ell\hat{W}^{\ell n}_t)dt\Big]
    +\E\Big[\int_{t_0}^\infty e^{-t}h\cdot\hat{Q}^n_tdt\Big]\\
    &\geq\E\Big[\int_0^{t_0} e^{-t}\sum_{\ell=1}^Lh^\ell(\hat{Q}^{\ell n}_t-\mu^\ell\hat{W}^{\ell n}_t)dt\Big]
    +\E\Big[\int_0^{t_0} e^{-t}\hat{W}^{\tot,n}_tdt\Big]\nonumber
\end{align}
where the last inequality is due to the nonnegativity of $\hat W^{\ell n}_t$
and the fact $h^\ell\mu^\ell\ge h^1\mu^1=1$.
Upon denoting
\begin{align*}
\eps^{t_0}(a^n,\beta^n)=\E\Big[\sum_{\ell=1}^Lh^\ell\big\|\hat{Q}^{\ell n}-\mu^\ell\hat{W}^{\ell n}\big\|_{t_0}\Big],
\qquad
\bar\eps^{t_0}(a^n)=\sup_{\beta^n\in\B^n}\eps^{t_0}(a^n,\beta^n),
\end{align*}
and recalling that by Lemma \ref{lem:W,beta ind},
$\hat{W}^{\tot,n}$ does not depend on the strategy, we have
\begin{align}\label{eq:infC geq}
    \inf_{\beta^n\in\B^n}C^n(\beta^n(a^n),a^n)\geq\E\Big[\int_0^{t_0} e^{-t}\hat{W}^{\tot,n}_tdt\Big]-\bar\eps^{t_0}(a^n).
\end{align}

We now assume that the SC applies the strategy $\beta^{*n}$.
Using the inequality
\[
    \sum_{\ell=1}^Lh^\ell\mu^\ell\hat{W}^{\ell n}_t\le\hat{W}^{1 n}_t+c\sum_{\ell=2}^L\hat{W}^{\ell n}_t
    \le \hat{W}^{\tot,n}_t+c\sum_{\ell=2}^L\hat{W}^{\ell n}_t
\]
in \eqref{eq:C=sum1}, we obtain
$C^n(\beta^{*n}(a^n),a^n)\leq\E[\int_0^{t_0}e^{-t}\hat{W}^{\tot,n}_tdt]+\tilde{\eps}^{t_0}(a^n)$,
where
\begin{align*}
    &\tilde{\eps}_1^{t_0}(a^n)=c\E\Big[\sum_{\ell=1}^L\big\|\hat{Q}^{\ell n}-\mu^\ell\hat{W}^{\ell n}\big\|_{t_0}\Big],
    &&\tilde{\eps}_2^{t_0}(a^n)=c\E\Big[\sum_{\ell=1}^L\int_{t_0}^\infty e^{-t}\|\hat{Q}^{\ell n}\|_{t}dt\Big],\\
    &\tilde{\eps}_3^{t_0}(a^n)=c\E\Big[\Big\|\sum_{\ell=2}^L\hat{W}^{\ell n}\Big\|_{t_0}\Big],
    &&\tilde{\eps}^{t_0}(a^n)=\sum_{k=1}^3\tilde\eps_k^{t_0}(a^n).
\end{align*}
Again using the fact that $\hat{W}^{\tot,n}$ is independent of $\beta^n$,
combining with \eqref{eq:infC geq} we obtain
\begin{align*}
\Del^n\geq-\tilde{\eps}^{t_0}(a^n)-\bar\eps^{t_0}(a^n).
\end{align*}
In view of Lemma \ref{lem:W,Q UI}, the
uniform (in $\beta^n$) convergence stated in Lemma \ref{lem:Q-W to 0}
shows that $\lim_n\bar\eps^{t_0}(a^n)=0$ for every $t_0$.
For the same reason, $\lim_n\tilde\eps^{t_0}_1(a^n)=0$.
Similarly, by Lemmas \ref{lem:W,Q UI} and \ref{lem:W To 0},
$\lim_n\tilde\eps^{t_0}_3(a^n)=0$. This shows that
\[
\liminf_n\Del^n\ge-\limsup_n\tilde\eps^{t_0}_2(a^n)
\ge-c\int_{t_0}^\iy e^{-t}(t+1)dt,
\]
where Lemma \ref{lem:W,Q UI} is used again. The result follows
on taking $t_0\to\iy$.
\end{proof}

\subsection{Proof of the Lemmas}\label{sec:lems}

The proofs require some preparation.
Denote the centered version of job size by $\bar{J}^{\ell n}_k=J_k^{\ell n}-\bxi_k^{\ell n}$. Recalling the definition of $\bb^{\ell n}$
from Assumption \ref{ass:Pi}.\ref{it:K,b} and \eqref{33}, we can write
\begin{align}\label{32}
    W^{\ell n}_t=\sum_{k=1}^{nt}\bar{J}_k^{\ell n}+\sum_{k=1}^{nt}\frac{1}{\mu^{\ell n}}+\sum_{k=1}^{nt}\frac{\bb^{\ell b}_k}{\sqrt{n}}-nB^{\ell n}_{t}=X_t^{\ell n}+Y_t^{\ell n}+\floornt\frac{1}{\mu^\ell}-nB^{\ell n}_{t},
\end{align}
where we defined
\begin{align}\label{32+}
  &X_t^{\ell n}=\sum_{k=1}^{nt}\bar{J}_k^{\ell n},
  &Y_t^{\ell n}=\sum_{k=1}^{nt}\frac{\bb^{\ell n}_k}{\sqrt{n}}.
\end{align}
By compactness of $K^\ell$ and the convergence $K^{\ell n}\to K^\ell$
in $d_{\rm H}$, the sets $K^{\ell n}$ are uniformly bounded.
Hence so are the processes $\bb^{\ell n}, \bq^{\ell n}$.
Recalling that by our convention $\hat Y^{\ell n}=n^{-1/2}Y^{\ell n}$, we have
\begin{align}\label{eq:bnd on Y}
    \|\hat Y^{\ell n}\|_{t}\leq ct\qquad {\rm and}\qquad w_t(\hat Y^{\ell n},\del)\leq c(\del+n^{-1}),
    \qquad 0\le \del\le t.
\end{align}
Second, with $e_t^n=\floornt-nt$, by \eqref{eq:WR=Gamma} we write $W^{\tot,n}_t=X_t^{\tot,n}+Y_t^{\tot,n}+e_t^n+ R_t^{\tot,n}$ and obtain
\begin{align}\label{eq:W tot}
    (\hat{W}^{\tot,n},\hat{R}^n)=\Gam(\hat{X}^{\tot,n}+\hat{Y}^{\tot,n}+\hat{e}^n).
\end{align}

\begin{lemma}\label{lem: X is UI}
There exists a constant $c$ such that for all $n\in\N$, $\ell\in[L]$, $a^n\in\calA^n$ and $\beta^n\in\B^n$,
\begin{enumerate}
\item The processes $\hat{X}^{\ell n}$ are martingales with respect to $\{\calG_t^n\}$, of quadratic covariation
\begin{align*}
    [\hat{X}^{\ell n},\hat{X}^{\ell' n}]_t=\frac{1}{n}\sum_{k=1}^{nt}\bar{J}_k^{\ell n}\bar{J}_k^{\ell' n},
\end{align*}
\item $\E\big[\|\hat{X}^{\ell n}\|_t^2\big]\leq  ct$ and $\E\big[[\hat X^{\ell n}]_t\big]\le ct$.
\end{enumerate}
 
\end{lemma}
\begin{proof}
As the index $\ell$ does not play any role in the proof, it is fixed and omitted to simplify notation.
Note that $\{\hat{X}^n_t\}_{t\geq 0}$ is adapted to $\{\calG_t^n\}_{t\geq 0}$ by the definition of $\hat{X}^n$ and of $\calG^n$.
Consider $s<t$. If $\lfloor nt\rfloor=\lfloor ns\rfloor$ then the increment of
$\hat X^n$ is zero. If $\lt\lfloor nt\rt\rfloor=\lt\lfloor ns \rt\rfloor +1$ then
\begin{align*}
    \E[\hat{X}^n_t-\hat{X}^n_s|\calG_s^n]&=\frac{1}{\sqrt{n}}\E[\bar{J}^n_{\lt\lfloor nt\rt\rfloor}|\calF^n_{\lt\lfloor nt\rt\rfloor-1}]=0.
\end{align*}
For general $t>s$ the martingale property follows by the tower property.
The expression for the quadratic covariation follows directly from its definition.
This proves part 1.

To prove part 2, abbreviate $\hat X^{\ell n}_t$ to $\hat X_t$ and $\bar{J}^{\ell n}$
to $\bar{J}_k$ in the next display to get
\begin{align}\label{306}
    \E[\hat{X}_t^2]&=\frac{1}{n}\sum_{k=1}^{nt}\expect{\bar{J}_k^2}+\frac{2}{n}\sum_{1\le m<k\le nt}\expect{\bar{J}_k\bar{J}_m}
    \\
    &\leq ct+\frac{2}{n}\sum_{1\le m<k\le nt}\expect{\E\lt[\bar{J}_k\middle|\calF_{k-1}^n\rt]\bar{J}_m}= ct,
    \label{307}
\end{align}
where the boundedness of the second moment of $\bar{J}_k^n$ follows from the boundedness of $\bq_k^{\ell n}$.
The claim regarding $\E[\|\hat{X}\|_t^2]$ now follows from Doob's inequality. Next, $\E\big[[\hat X^{\ell n}]_t\big]$
is equal to the first term in \eqref{306}, and thus the required estimate follows from \eqref{307}.
\end{proof}

\begin{proof}[Proof of Lemma \ref{lem:W,Q UI}]
In this proof the superscript $\tot$ is omitted. Using \eqref{eq:W tot},
the explicit expression for $\Gam$ given in \eqref{e02}, and \eqref{eq:bnd on Y},
\begin{align}\label{eq:a1}
    \|\hat{W}^n\|_t\leq 2\big(\| \hat{X}^n\|_t+\| \hat{Y}^n\|_t\big)+1\leq2\| \hat{X}^n\|_t+c(t+1).
\end{align}
Hence by Lemma \ref{lem: X is UI}, $\E\big[\|\hat{W}^n\|_t^2\big]\leq c(t+1)^2$.

Next, by the definition of $D_t^n$ we have $W_t^{\ell n}=\sum_{k=D^{\ell n}_t+1}^{nt}J_k^{\ell n}$, where clearly $D^{\ell n}_t\le nt$.
By the definition of $\bb_k^{\ell n}$,
\begin{align*}
    \hat{W}_t^{\ell n}=\frac{1}{\sqrt{n}}\sum_{k=D_t^{\ell n}+1}^{nt}\brac{J^{\ell n}_k-\bxi_k^{\ell n}+\frac{\bb_k^{\ell n}}{\sqrt{n}}+\frac{1}{\mu^\ell}}=\hat{X}^{\ell n}_t+\hat{Y}_t^{\ell n}-\hat{X}^{\ell n}_{D_t^{\ell n}/n}-\hat{Y}^{\ell n}_{D_t^{\ell n}/n}+\frac{\hat{Q}_t^{\ell n}}{\mu^\ell}.
\end{align*}
Hence,
\begin{align}\label{eq:b2}
    \hat{Q}^{\ell n}_t=\mu^\ell(\hat{W}^{\ell n}_t-\hat{X}^{\ell n}_t-\hat{Y}_t^{\ell n}+\hat{X}^{\ell n}_{D_t^{\ell n}/n}+\hat{Y}^{\ell n}_{D_t^{\ell n}/n}).
\end{align}
By the bound just obtained on the second moment of $\hat W^n$
and the bounds used in the previous paragraph on $\hat X^n$ and $\hat Y^n$,
$\E\big[\|\hat{Q}^n\|_t^2\big]\leq c(t+1)^2$.
\end{proof}


\begin{lemma}\label{lem:X tight}
For any control $a^n\in\calA^n$, the sequence $\{\hat{X}^{\ell n},[\hat{X}^{\ell n}],\hat{Y}^{\ell n},\hat{W}^{\tot,n},\hat{R}^n\}_{n\in\N}$ is uniformly $\calC$-tight.
\end{lemma}

$\calC$-tightness is not stated, and in general does not hold,
for each component $\hat W^{\ell n}$.

\begin{proof}
It suffices to prove uniform $\calC$-tightness for each component.
We prove first for each $\ell\in[L]$ that the first three components of the tuple are uniformly $\calC$-tight, omitting the index $\ell$ in the proof.
By \eqref{eq:bnd on Y}, the process $\hat{Y}^n$ is bounded by the function $ct$, and $w_{t_0}(\hat{Y}^n,\delta)\leq c(\delta+n^{-1})$.
Because both bounds are uniform in $\beta^n\in\B^n$, uniform $\calC$-tightness follows.

Turning to $\hat{X}^n$, given $t_0$, the uniform bound needed on $\PP(\|\hat{X}^n\|_{t_0}>x)$ follows from Lemma \ref{lem: X is UI}.
Next, for any $\delta$ of the form $3t_0/N$, some $N\in\N$,
consider the $N=3t_0/\delta$ closed intervals $\Del t_i=[(i-1)\frac{\del}{3},i\frac{\del}{3}]$.
Then 
\begin{align*}
    \PP\Big(\sup_{t,s\leq t_0,|t-s|<\delta}|\hat{X}^n_t-\hat{X}^n_s|>\eps\Big)&\leq\PP\Big(\exists i\in\{1,...,N\}:\sup_{t,s\in\Del t_i}|\hat{X}^n_t-\hat{X}^n_s|>\tfrac{\eps}{3}\Big)\\
    &\leq\sum_{i=1}^N\PP\Big(\sup_{t,s\in\Del t_i}|\hat{X}^n_t-\hat{X}^n_s|>\tfrac{\eps}{3}\Big)\\
    &\leq\sum_{i=1}^N\frac{\E\big[\sup_{t,s\in\Del t_i}|\hat{X}^n_t-\hat{X}^n_s|^4\big]}{(\eps/3)^4}.
\end{align*}
Denoting $t_i=(i-1)\frac{\del}{3}$,
\begin{align*}
    \sup_{t,s\in\Del t_i}|\hat{X}^n_t-\hat{X}^n_s|^4\leq c\sup_{t\in\Del t_i}|\hat{X}^n_t-\hat{X}^n_{t_i}|^4.
\end{align*}
Since $\hat{X}^n_t$ is a martingale, Burkholder's inequality yields
\begin{align}\label{305}
    \E\Big[\sup_{t\in\Del t_i}|\hat{X}^n_t-\hat{X}^n_{t_i}|^4\Big]\leq c\E\big[\big([\hat{X}^n]_{t_{i+1}}-[\hat{X}^n]_{t_i}\big)^2\big]=\frac{c}{n^2}\sum_{k,m=\floor{nt_i}+1}^{nt_{i+1}}\E[(\bar{J}_k^n\bar{J}_m^n)^2]\leq \frac{c}{n^2}(n\delta+1)^2,
\end{align}
where the last ineqaulity combines Cauchy-Schwarz inequality, Assumption \ref{ass:Pi}.\ref{it:4th moment} and the fact that $|nt-nt_i|\leq n\delta$.
Note that $c$ does not depend on $\beta^n$.
Combining the above estimates gives
\begin{align*}
    \sup_{\beta\in\B^n}\PP(w_{t_0}(\hat{X}^n,\delta)>\eps)\leq\frac{3t_0}{\delta}\cdot\frac{c}{\eps^4}(\delta+n^{-1})^2.
\end{align*}
This converges to $0$ when taking $n\to\infty$ and then $\delta\to0$, proving uniform $\calC$-tightness of $\{\hat{X}^{\ell n}\}_n$.

As for $[\hat X^n]$, recalling that this process is nondecreasing,
we have by Lemma \ref{lem: X is UI} the estimate needed on $\PP([\hat X^n]_{t_0}>x)$. To estimate its modulus
of continuity, a similar consideration as for $\hat X^n$ gives
\begin{align*}
    \PP\Big(\sup_{t,s\leq t_0,0<t-s<\delta}([\hat{X}^n]_t-[\hat{X}^n]_s)>\eps\Big)
    &\leq\sum_{i=1}^N\frac{\E\big[([\hat{X}^n]_{t_{i+1}}-[\hat{X}^n]_{t_i})^2\big]}{(\eps/3)^2}.
\end{align*}
Using \eqref{305} to bound the terms in the above sum as we did for $\hat X^n$
yields uniform $\calC$-tightness of $[\hat X^n]$ as in the
argument for $\hat X^n$.

Next, by \eqref{eq:WR=Gamma} we have $(\hat{W}^{\tot,n},\hat{R}^n)=\Gam(\hat{X}^{\tot,n}+\hat{Y}^{\tot,n}+\hat{e}^n)$. Hence the continuity of $\Gam$
implies uniform $\calC$-tightness of $\{\hat{W}^{\tot,n},\hat{R}^n\}_n$.
\end{proof}

\begin{proof}[Proof of Lemma \ref{lem:Q-W to 0}]
	All the processes in the proof are associated with a fixed $\ell$, omitted
	from the notation. Write
	\begin{align*}
		W_t^n=\sum_{k=1}^{nt}J^n_k-nB^n_t=\sum_{k=D^n_t+1}^{nt}J^n_k-j^n_t,
	\end{align*}
	where $j_t^n=j^{\ell n}_t$ is the work that the server has completed by time $t$ of the job served at that time, if it is of type $\ell$ (0 otherwise).
	Denoting $\hat{\psi}^{n}=\hat{X}^n+\hat{Y}^n$,
	\begin{align*}
		\hat{W}^n_t&=\frac{1}{\sqrt{n}}\sum_{k=D^n_t+1}^{nt}\brac{\bar{J}^n_k+\bxi_k^n}-\frac{j^n_t}{\sqrt{n}}\\
		&=\frac{1}{\sqrt{n}}\sum_{k=D^n_t+1}^{nt}\brac{\bar{J}^n_k+\frac{\bb^n_k}{\sqrt{n}}}+\frac{1}{\sqrt{n}}\sum_{k=D^n_t+1}^{nt}\frac{1}{\mu}-\frac{j^n_t}{\sqrt{n}}\\
		&=\hat{\psi}^n(t)-\hat{\psi}^n\brac{\tfrac{D^n_t}{n}}+\frac{\hat{Q}^n_t}{\mu}-\frac{j^n_t}{\sqrt{n}},
	\end{align*}
	where on the last line we used \eqref{eq:Ql}.
	The last display readily gives two bounds, namely
	\begin{align}\label{eq:bound W-Q 1}
		\|\hat Q^n\|_{t_0}\leq\|\hat W^{\tot,n}\|_{t_0}+2\|\hat{\psi}^n\|_{t_0}+n^{-1/2}\|j^n\|_{t_0},
	\end{align}
	and, using again the identity \eqref{eq:Ql},
	\begin{align}\label{eq:bound W-Q}
		\Big\|\hat{W}^n-\frac{\hat{Q}^n}{\mu}\Big\|_{t_0}\leq w_{t_0}\brac{\hat{\psi}^n,\frac{\norm{Q^n}_{t_0}}{n}}+\frac{\norm{j^n}_{t_0}}{\sqrt{n}}.
	\end{align}
	We show that the RHS of
	\eqref{eq:bound W-Q} converges to 0 uniformly in probability, as $n\to\iy$.
	To this end, bound the second term as
	\begin{align*}
		\frac{\norm{j^n}_{t_0}}{\sqrt{n}}&\leq\max_{1\leq k\leq nt_0}\frac{\abs{J^n_k}}{\sqrt{n}}\\
		&\leq \max_{1\leq k\leq nt_0}\frac{|\bar J^n_k|}{\sqrt{n}}+\frac{\bxi^n_k}{\sqrt{n}}\\
		&\leq w_{t_0}(\hat{X}^n,\delta)+\frac{c}{\sqrt{n}},
	\end{align*}
	for any $\delta>0$.
	The uniform $\calC$-tightness of $\hat X^n$ stated in Lemma \ref{lem:X tight}
	implies that the above expression converges to 0 uniformly in probability
	as $n\to\iy$ and then $\delta\to0$.
	As for the first term in \eqref{eq:bound W-Q}, note that for $\eps>0$,
	\begin{align*}
		\PP\bigg(w_{t_0}\brac{\hat{\psi}^n,\frac{\norm{Q^n}_{t_0}}{n}}>\eps\bigg)&
		\leq\PP\big(w_{t_0}(\hat{\psi}^n,\delta)>\eps\big)+\PP\bigg(\frac{\|Q^n\|_{t_0}}{n}>\delta\bigg).
	\end{align*}
	Lemma $\ref{lem:X tight}$ implies that
	$\lim_{\delta\to 0}\limsup_{n\to\infty}\sup_{\beta\in\B^n}\PP(w_{t_0}(\hat{\psi}^n,\delta)>\eps)=0$. Also, by \eqref{eq:bound W-Q 1} and uniform $\calC$-tightness of $\hat W^{\tot,n}$ and $\hat{\psi}^n$, it follows that $n^{-1}\|Q^n\|_{t_0}\to0$ uniformly
	in probability. Altogether, these estimates show that
	$\|\hat W^n-\frac{\hat Q^n}{\mu}\|_{t_0}\to 0$
uniformly in probability, as $n\to\infty$.
\end{proof}

\begin{proof}[Proof of Lemma \ref{lem:W To 0}]
Let the strategy $\beta^{*n}$ be used by the SC.
Fix $\eps,t_0>0$ and denote by $\calE^n$ the event $\{\|\sum_{\ell=2}^L\hat{W}^{\ell n}\|_{t_0}>\eps\}$.
Denote by $\hat{\psi}^n$ the process
\begin{align*}
    \hat{\psi}^n_t=\sum_{\ell=2}^L\brac{\frac{\hat{e}^n(t)}{\mu^\ell}+\hat{X}^{\ell n}_t+\hat{Y}_t^{\ell n}}.
\end{align*}
Then, by \eqref{32},
\begin{align}\label{eq: W by F}
    \sum_{\ell=2}^L\hat{W}^{\ell n}_t=\hat{\psi}^n_t+\sqrt{n}t\sum_{\ell=2}^L\frac{1}{\mu^\ell}-n\sum_{\ell=2}^L\hat{B}_t^{\ell n}.
\end{align}
By Lemma \ref{lem:X tight} and the convergece $\hat{e}^n\to0$, $\{\hat{\psi}^n\}_n$ is $\calC$-tight.
Let
\begin{align*}
    &\tau=\tau^n=\inf\lt\{t\in[0,t_0]:\sum_{\ell=2}^L\hat{W}^{\ell n}_t>\eps\rt\},\\
    &\theta=\theta^n=\sup\lt\{t\in[0,\tau):\sum_{\ell=2}^L\hat{W}^{\ell n}_t<\frac{\eps}{2}\rt\}.
\end{align*}
Let $j^n_t$ denote the residual work of the job in service at time $t$.
On $\calE^n$, $\tau$ and $\theta$ take values in $[0,t_0]$.
In addition, $\sum_{\ell=2}^L(nB_\tau^{\ell n}-nB_\theta^{\ell n})\geq n(\tau-\theta)-j^n_\theta$ on $\calE^n$ because there is work to serve throughout this time interval (thus the server is not idle), and moreover, there is work of high-priority ($\ge2$) types so the cumulative effort of serving type 1 jobs is bounded by $j^n_\theta$ (in case the job in service at time $\theta$ is of type 1).
Therefore by \eqref{eq: W by F} and the heavy-traffic condition $\sum_{\ell=1}^L(\mu^\ell)^{-1}=1$, on $\calE^n$ one has
\begin{align*}
    \frac{\eps}{2}\leq\sum_{\ell=2}^L\brac{\hat{W}_{\tau}^{\ell n}-\hat{W}^{\ell n}_\theta}&=\hat{\psi}^n_\tau-\hat{\psi}^n_\theta+\sqrt{n}\brac{\tau-\theta}\brac{1-\frac{1}{\mu^1}}-\sqrt{n}\brac{\tau-\theta}+\frac{j^n_\theta}{\sqrt{n}}\\
    &=\hat{\psi}^n_\tau-\hat{\psi}^n_\theta-\frac{\sqrt{n}\brac{\tau-\theta}}{\mu^1}+\frac{j^n_\theta}{\sqrt{n}}.
\end{align*}
Denoting $\calE^n_1=\{\hat{\psi}^n_\tau-\hat{\psi}^n_\theta-\frac{(\tau-\theta)\sqrt{n}}{\mu^1}>\frac{\eps}{4}\}$,
\begin{align*}
    \prob{\calE^n}
    &\leq\PP(\calE^n_1)+\prob{\frac{j^n_\theta}{\sqrt{n}}>\frac{\eps}{4}}.
\end{align*}
To show that the second term converges to 0, note that it is bounded by $n^{-1/2}\max_{k\in[nt],\ell\in[L]}J^{\ell n}_k$.
Hence the argument given in the proof of Lemma \ref{lem:Q-W to 0} applies here.
As for the first term,
\begin{align*}
\PP(\calE^n_1)
&=\PP(\calE^n_1, \tau-\theta\geq n^{-1/4}) + \PP(\calE^n_1, \tau-\theta< n^{-1/4})
\\
    &\leq\prob{\hat{\psi}^n_\tau-\hat{\psi}^n_\theta-\frac{n^{1/4}}{\mu^1}>\frac{\eps}{4}}+\prob{w_{t_0}\brac{\hat{\psi}^n,n^{-1/4}}>\frac{\eps}{4}}\\
    &\leq\prob{\|\hat{\psi}^n\|_{t_0}>cn^{1/4}}+\prob{w_{t_0}\brac{\hat{\psi}^n,n^{-1/4}}>\frac{\eps}{4}}.
\end{align*}
Both terms converge to 0 by the $\calC$-tightness of $\{\hat{\psi}^n\}_n$.
This shows that $\PP(\calE^n)\to0$ and completes the proof.
\end{proof}

Theorem \ref{thm:cmu is opt} justifies the choice of the strategy
$\beta^{*n}$ by the SC, and this will be assumed for the rest of the paper.

\section{The MCP}
\label{sec:mcp}

In this section we prove Proposition \ref{pr:verification}.
The proof is based on existence of classical solutions to fully nonlinear
uniformly elliptic equations
with Dirichlet boundary conditions. Specifically, consider the following version of the HJB
equation, set on a finite interval, namely
\begin{align}\label{eq:HJB x}\tag{HJB-$x$}
    \sum_{\ell=1}^L\max_{(b^\ell,q^\ell)\in\text{ch}(K^\ell)}\lt\{q^\ell u''(w)+b^{\ell}u'\brac{w}\rt\}-u(w)+w=0,\quad0<w< x,
\end{align}
with the boundary conditions $u(0)=\ph_0$, $u(x)=\ph_1$.
Let us show that the assumptions of \cite[Theorem~17.18]{gilbarg98} hold.
To this end, the infimum in the formulation given in \cite{gilbarg98}
can be translated to supremum by the transformation $u\mapsto-u$ in
an obvious way. Moreover, as indicated in the paragraph that follows the statement of
\cite[Theorem~17.18]{gilbarg98}, it suffices to verify the conditions
in \cite[(17.62)]{gilbarg98}. To verify these conditions, note that
uniform ellipticity holds in our case by the assumption that $K^\ell$
is compact and contained in $\R\times(0,\iy)$. The boundedness of the coefficients
and their derivatives is trivial in our case, as these coefficients are given,
in the notation of \cite{gilbarg98}, as follows: the coefficients $a^{ij}_\nu$,
$b^i_\nu$ and $c_\nu$ are constants for each $\nu$ (the role of $\nu$
is played in our setting by the point $(b^\ell,q^\ell)\in\R^2$),
and $f_\nu(w)=w$. We thus deduce that there exists a unique classical solution to
\eqref{eq:HJB x}, for any $(\ph_0,\ph_1)\in\R^2$.
The proof now proceeds in several steps.

\noi{\bf Step 1.}
We prove that the value function $\sfV$ satisfies the boundary conditions of \eqref{eq:HJB},
namely $\sfV'(0)=0$ and $\limsup_{w\to\iy}|\sfV(w)|/w<\iy$.

For the boundary condition at infinity,
note that by Burkholder's inequality and property \ref{it:dXy in chK} of the MCP
one has $\bar\E[\|\sfX^\tot\|_t]\leq c(\bar\E[\sfX^\tot]_t])^{1/2}<ct^{1/2}$.
Also by property 4, $\|\sfY^\tot\|_t\le ct$.
By property 3, $\|\sfW^\tot\|_t\le 2(w+\|\sfX^\tot\|_t+\|\sfY^\tot\|_t)$.
Hence, for every $w$ and $\calS\in\frS_w$,
\begin{align*}
	0\le\sfC(\calS)&=\bar\E\Big[\int_0^\infty e^{-t}\sfW^\tot_tdt\Big]
	\\
	&\leq\int_0^\infty e^{-t}\bar\E[\|\sfW^\tot\|_t]dt\\
	&\leq c\int_0^\infty e^{-t}(w+t^{1/2}+t)dt.
\end{align*}
Since the constant $c$ does not depend on $\calS$, it follows that
\begin{equation}\label{115}
0\le\sfV(w)\le c(w+1), \qquad w\in[0,\iy).
\end{equation}
Next, to show $\sfV'(0)=0$, note that
\begin{align*}
	\sfV(w)-\sfV(0)&=\sup_{\calS\in\frS_w}\bar\E\Big[\int_0^\infty e^{-t}\sfW^\tot_tdt\Big]-\sup_{\calS\in\frS_0}\bar\E\Big[\int_0^\infty e^{-t}\sfW^\tot_tdt\Big]\\
	&=\sup_{\calS\in\frS_w}\bar\E\Big[\int_0^\infty e^{-t}\Gam_1(w+\sfX^\tot+\sfY^\tot)(t)dt\Big]-\sup_{\calS\in\frS_0}\bar\E\Big[\int_0^\infty e^{-t}\Gam_1(\sfX^\tot+\sfY^\tot)(t)dt\Big]\\
	&\leq
	\sup_{\calS\in\frS_w}\bar\E\Big[\int_0^\infty e^{-t}\big(\Gam_1(w+\sfX^\tot+\sfY^\tot)(t)-\Gam_1(\sfX^\tot+\sfY^\tot)(t)\big)dt\Big].
\end{align*}
We shall use the following property of the map $\Gam_1$, which can be easily checked
directly by its definition.
Let $w\ge0$, $\ph\in \calC_\R[0,\iy)$, $\psi^1=\Gam_1(\ph)$ and $\psi^2=\Gam_1(w+\ph)$.
Denote the hitting time $\tau=\tau(w+\ph)=\inf\{t\ge0:w+\ph_t\le0\}$. Then for all $t$, $\psi^1_t\le\psi^2_t$,
whereas for all $t\ge\tau$, $\psi^1_t=\psi^2_t$.
Hence letting $\tau=\tau(w+\sfX^\tot_t+\sfY^\tot_t)
=\inf\{t\geq0:w+\sfX^\tot_t+\sfY^\tot_t\leq0\}$, we have
\begin{align*}
	\sfV(w)-\sfV(0)\leq\sup_{\calS\in\frS_w}\bar\E\Big[\int_0^\tau e^{-t}\big(\Gam_1(w+\sfX^\tot+\sfY^\tot)(t)-\Gam_1(\sfX^\tot+\sfY^\tot)(t)\big)dt\Big].
\end{align*}
Next, using the Lipschitz property of $\Gam_1$, we bound further
\begin{align}\label{120}
	0\le\frac{\sfV(w)-\sfV(0)}{w}\leq\frac{1}{w}\sup_{\calS\in\frS_w}\bar\E\Big[\int_0^\tau e^{-t}cwdt\Big]\leq c\sup_{\calS\in\frS_w}\bar\E[\tau\land 1].
\end{align}
Because for all $w$ and all $\calS\in\frS_w$, the processes $\sfY^\tot$ are uniformly Lipschitz,
denoting by $c_0$ the Lipschitz constant, we have
\begin{align*}
	\PP(\tau>t)
	&\leq\PP\big(\inf_{0\leq s\leq t}\{\sfX_s^\tot+c_0s\}>-w\big).
\end{align*}
To proceed, note that by Property 4 of the MCP and the fact that $K^\ell\subset\R\times(0,\infty)$,
the mapping $s\mapsto[\sfX^\tot]_s$ has the property
\[
c_1\le\frac{[\sfX^\tot]_{s_2}-[\sfX^\tot]_{s_1}}{s_2-s_1}\le c_2
\]
for all $0\le s_1<s_2$, where $0<c_1<c_2$ are constants.
Consequently, this mapping admits a Lipschitz, strictly increasing inverse, denoted $T$.
By the time-change theorem for martingales \cite[Theorem 3.4.6]{karatzas2014}, the process $Z_t=\sfX^\tot_{T(t)}$ is a SBM. Thus
\begin{align*}
	\PP(\tau>t)&\leq
	\PP\big(\inf_{0\leq s\leq t}\{Z_{[\sfX^\tot]_s}+c_0s\}>-w\big)\\
	&\le\PP\big(\inf_{0\leq s\leq [\sfX^\tot]_t}\{Z_s+c_0T(s)\}>-w\big)\\
	&\leq\PP\big(\inf_{0\leq s\leq c_1t}\{Z_s+cs\}>-w\big).
\end{align*} 
The RHS above does not depend on $\calS$, and, for every $t>0$, converges to zero
as $w\to0$. This shows that the RHS of \eqref{120} converges to zero as $w\to0$,
which shows that $\sfV'(0)=0$.

\noindent{\bf Step 2.}
Denote by $f=f^{(x)}$ the unique solution, defined on $[0,x]$,
to \eqref{eq:HJB x} with the boundary conditions
$f(0)=\sfV(0)$, $f(x)=\sfV(x)$.
In this step we relate this solution to a control problem set on the bounded domain $[0,x]$.
Given any admissible control system for the MCP, $\calS$, let 
\[
\tau_x=\inf\{t\geq 0:\sfW_t^\tot\notin (0,x)\}.
\]
For $0\leq t\leq \tau_x$, property 3 of the MCP
reduces to $\sfW^\tot=w+\sfX^\tot+\sfY^\tot$ and $\sfR=0$.
The goal of this step is to prove the identity
\begin{align}\label{eq:u a}
	f(w)=\sup_{\calS\in\frS_w}\bar\E\lt[\int_0^{\tau_x}e^{-s}\sfW_s^\tot ds+e^{-\tau_x}\indicator{\tau_x<\iy}\sfV\brac{\sfW^\tot_{\tau_x}}\rt],\quad 0\le w\le x.
\end{align}
Abbreviating $\sfW^\tot$ to $\sfW$,
applying It\^o's lemma and denoting $\tau=\tau(x,t):=\tau_x\land t$,
\begin{align*}
	e^{-\tau}f(\sfW_{\tau})=f(w)-\int_0^{\tau}e^{-s}f(\sfW_s)ds+\int_0^{\tau}e^{-s}f'(\sfW_s)d\sfW_s+\frac{1}{2}\int_0^{\tau}e^{-s}f''(\sfW_s)d[\sfX^\tot]_s.
\end{align*}
Substitute $d\sfW_s=d\sfX^\tot_s+d\sfY^\tot_s$ and $d[\sfX^\tot]_s=\sum_{\ell=1}^Ld[\sfX^{\ell}]_s$ (since, by property 4 of the MCP, $[\sfX^\ell,\sfX^{\ell'}]=0$ for $\ell\neq\ell'$) and take expectation.
Using the fact that $\sfX^\tot$ is a martingale gives
\begin{align}\label{100}
	&\bar\E\lt[e^{-\tau}f\brac{\sfW_{\tau}}\rt]\\
	\notag
	&\quad=f(w)-\bar\E\bigg[\int_0^{\tau}e^{-s}f\brac{\sfW_s}ds-\sum_{\ell=1}^L\Big(\int_0^{\tau}e^{-s}f'(\sfW_s)d\sfY^\ell_s+\frac{1}{2}\int_0^{\tau}e^{-s}f''(\sfW_s)d[\sfX^{\ell}]_s\Big)\bigg].
\end{align}
Next, by Properties 2 and 4 of the MCP, the sample paths of $\sfY^\ell$ and $[\sfX^\ell]$ are absolutely continuous.
Denote by $\frac{d\sfY_s^\ell}{ds}$ and $\frac{d[\sfX^\ell]_s}{ds}$ their respective a.e.\  derivatives.
Then b Property 4 of the MCP, for a.e.\ $s$,
\[
\Big(\frac{d\sfY_s^\ell}{ds},\frac{d[\sfX^\ell]_s}{2ds}\Big)\in\text{ch}(K^\ell).
\]
As a result, for every $\ell$,
\begin{align*}
	\int_0^{\tau}e^{-s}f'(\sfW_s)d\sfY^\ell_s+\frac{1}{2}\int_0^{\tau}e^{-s}f''(\sfW_s)d[\sfX^{\ell}]_s
	&=
	\int_0^{\tau}e^{-s}\Big(f'(\sfW_s)\frac{d\sfY^\ell_s}{ds}
	+\frac{1}{2}f''(\sfW_s)\frac{d[\sfX^{\ell}]_s}{ds}\Big)ds
	\\
	&\le
	\int_0^{\tau}e^{-s}\max_{(b^\ell,q^\ell)\in\text{ch}(K^\ell)}
	\Big(f'(\sfW_s)b^\ell
	+f''(\sfW_s)q^\ell\Big)ds.
\end{align*}
Using this in \eqref{100} gives
\begin{align}\label{eq:bnd by max1}
	&\E\lt[e^{-\tau}f\brac{\sfW_{\tau}}\rt]\\
	&\qquad\leq f(w)-\E\bigg[\int_0^{\tau}e^{-s}f\brac{\sfW_s}ds-\sum_{\ell=1}^L\int_0^{\tau}e^{-s}\max_{(b^\ell,q^\ell)\in\text{ch}(K^\ell)}\lt\{f''\brac{\sfW_s}q^\ell+b^{\ell}f'\brac{\sfW_s}\rt\}ds\bigg]\nonumber\\
	&\qquad=f(w)-\E\bigg[\int_0^{\tau}e^{-s}\sfW_sds\bigg],\nonumber
\end{align}
where the last equality holds because $f$ solves \eqref{eq:HJB x}.
On the event $\{\tau_x<t\}$, we have $\sfW_{\tau_x}\in\{0,x\}$, hence the boundary conditions of \eqref{eq:HJB x} yield
\begin{align}\label{114}
	f(w)\geq\bar\E\lt[\int_0^{\tau_x\land t}e^{-s}\sfW_sds+e^{-\tau_x}\indicator{\tau_x<t}\sfV\brac{\sfW_{\tau_x}}+e^{-t}\indicator{\tau_x\geq t}f(\sfW_t)\rt].
\end{align}
We now take the limit $t\to\infty$.
Note that for $s\le \tau_x\w t$, $\sfW_s$ is bounded by $x$, and the function
$u_x$ is also bounded. Using the bounded convergence theorem for the first and last terms, and, noticing that the almost sure limit of $e^{-\tau_x}\indicator{\tau_x<t}\sfV\brac{\sfW_{\tau_x}}$ is $e^{-\tau_x}\indicator{\tau_x<\infty}\sfV\brac{\sfW_{\tau_x}}$ and using Fatou's lemma for this term, and finally taking supremum over all
admissible systems, we obtain that $f(w)$ is bounded below
by the RHS of \eqref{eq:u a}.


For the reverse inequality, we invoke a measurable selection theorem, according to which for each $\ell$ there exists
a function $(b^{\ell,*},q^{\ell,*}):[0,x]\to \text{ch}(K^\ell)$ such that
for every $y\in[0,x]$,
\begin{align}\label{736}
q^{\ell,*}(y)f''(y)+b^{\ell,*}(y)f'(y)=\max\{qf''(y)+bf'(y)
:(b,q)\in \text{ch}(K^\ell)\}.
\end{align}
This follows directly from \cite[Lemma 8.10]{budhiraja19}.
Denote $\sigma^{\ell,*}(y)=2\sqrt{q^{\ell,*}(y)}$.


The argument requires the construction of an admissible system
for the MCP, $\calS^*$. The first step in this construction will be to argue that there exists a weak solution $(\sfW,\sfR,Z,\{\bar\calF_t\})$ to the  SDE
\begin{align}\label{eq:SDE1}\tag{SDE1}
	\sfW_t=w+\sum_{\ell=1}^L\int_0^t\brac{\sigma^{{\ell},*}(\sfW_s)dZ_s^\ell+b^{\ell,*}(\sfW_s)ds}+\sfR_t,
\end{align}
where $w\ge0$, $\{\bar\calF_t\}$ is a filtration
$Z=(Z^\ell)$ is a standard $L$-dimensional $\{\bar\calF_t\}$-BM, $\sfR$ the reflection term at $0$,
and $(\sfW,\sfR)$ are adapted.
Denote $\hat\sigma(y)=(\sum_{\ell=1}^L(\sigma^{\ell,*}(y))^2)^{1/2},y\geq 0$.
Consider the SDE on $\R$
\begin{align*}
	\Psi_t=w+\int_0^t\hat\sigma(|\Psi_s|)d\tilde{Z}_s+\int_0^t\text{sign}\brac{\Psi_s}b^{\tot,*}\brac{|\Psi_s|}ds
\end{align*}
where $\tilde Z$ is a 1d BM, and $\text{sign}(y)=1$ for $y>0$ and $=-1$ for $y\le0$.
By \cite[Theorem 5.5.15]{karatzas2014} and the fact that $\hat\sig$ is bounded away from zero,
this SDE admits a weak solution $(\Psi,\tilde Z,\{\tilde\calF_t\})$. Denote $\sfW_t=|\Psi_t|$ and $\hat Z_t=\int_0^t\sign(\Psi_s)d\tilde Z_s$. Then clearly $\hat Z$ is an SBM, and, by  Tanaka's formula \cite[Propositions 9.2 and 9.3]{LeGall2016}, there exists a process $\sfR$ with sample paths in $\calC^+$ such that
\begin{align}\label{eq:111}
	\sfW_t&=w+\int_0^t\hat\sigma(\sfW_s)d\hat{Z}_s+\int_0^tb^{\tot,*}\brac{\sfW_s}ds+\sfR_t
\end{align}
and $\int_{[0,\iy)}\sfW_td\sfR_t=0$ a.s.
Notice that the term $\sum b^{\ell,*}(\sfW_s)ds$ of \eqref{eq:SDE1} agrees with
the term $b^{\tot,*}(\sfW_s)ds$ of \eqref{eq:111}. Hence the proof that \eqref{eq:SDE1}
possesses a weak solution will be complete once we construct
an $L$-dimensional SBM $(Z^\ell)$ that satisfies
\begin{align}\label{eq:112}
	\sum_{\ell=1}^L\sigma^{\ell,*}(\sfW_s)dZ^\ell_s=\hat\sigma(\sfW_s)d\hat Z_s.
\end{align}
To this end we look for a measurable map $U:(0,\iy)^L\to\R^{L\times L}$
that sends any $v\in\R_+^L\setminus\{0\}$ to a matrix whose columns comprise an orthonormal basis of $\R^L$.
Denote by $\{e_\ell\}$ the standard basis of $\R^L$. Given a column vector $v\in(0,\iy)^L$, the collection $\{v,e_1,...,e_{L-1}\}$ spans $\R^L$. Denote by $U(v)$ the matrix obtained by the Gram-Schmidt procedure, applied on the collection $\{v,e_1,...,e_{L-1}\}$.
Then the first column of the matrix $U(v)$ is equal to $v/\|v\|$, and 
\begin{equation}\label{113}
	v^TU(v)=\|v\|e_1^T.
\end{equation}
In addition, the map $U$ is measurable by the Gram-Schmidt construction.

Assume WLOG that the probability space supports $L-1$ independent SBMs, independent of $\hat Z$; denote the resulting $L$-dimensional SBM by $(\check Z^\ell)$ where  $\check Z^1=\hat Z$. Denote by $\sig^*(\sfW_t)$ the column-vector valued process
$(\sigma^{1,*}(\sfW_t),...,\sigma^{L,*}(\sfW_t))^T$.
Define the $\R^{L\times L}$-valued process $\eta_t=U(\sigma^*(\sfW_t))$.
Finally, let
\[
Z_t=\int_0^t\eta_sd\check Z_s.
\]
Because $\eta_t$ is an orthonormal matrix for every $t$,
it follows by \cite[Proposition 3.2.17]{karatzas2014} that $Z$ is an $L$-dimensional SBM.
It remains to show \eqref{eq:112}.
Note that the LHS of \eqref{eq:112} is $\sigma^*(\sfW_s)^TdZ_s$. By construction,
\begin{align*}
	\sigma^*(\sfW_s)^TdZ_s&=\sigma^*(\sfW_s)^T\eta_sd\check Z_s\\
	&=\sigma^*(\sfW_s)^TU(\sig^*(\sfW_s))d\check Z_s
	\\
	&=\|\sig^*(\sfW_s)\|d\check Z^1_s\\
	&=\hat \sig(\sfW_s)d\hat Z_s,
\end{align*}
where the third equality follows from \eqref{113} and the last one by the identities
$\hat\sig(\cdot)=\|\sig^*(\cdot)\|$ and $\check Z^1=\hat Z$.
This completes the proof that the tuple $(\sfW,\sfR,Z)$ solves \eqref{eq:SDE1}.

The next step in the construction of $\calS^*$ is to let
\[
\sfX^{\ell}_t=\int_0^t\sigma^{\ell,*}(\sfW_s)dZ^\ell_s,
\qquad
\sfY_t^{\ell}=\int_0^tb^{\ell,*}(\sfW_s)ds,
\qquad
\bar\calF_t=\tilde\calF_t\vee\sig\{\check Z^\ell_s:0\le s\le t,2\le\ell\le L\}.
\]
It remains to prove that $\calS$ satisfies Properties 1--4 of the MCP.
Properties 1--2 follow by construction.
For Property 3, by construction, \eqref{eq:SDE1} reads
$\sfW_t=w+\sfX^\tot_t+\sfY^\tot_t+\sfR_t$. Moreover, $\sfW_t$ is $\R_+$-valued
and $\int_{[0,\iy)}\sfW_td\sfR_t=0$ a.s. Hence, by Skorohod's lemma,
$(\sfW,\sfR)=\Gam(w+\sfX^\tot+\sfY^\tot)$ a.s.
As for Property 4, note first that for $\ell\neq\ell'$, one has $[Z^\ell,Z^{\ell'}]=0$, hence
$[\sfX^{\ell},\sfX^{\ell'}]=0$, using \cite[Proposition 3.2.17]{karatzas2014}.
Moreover, by definition, $(b^{\ell,*}(y),(\sigma^{\ell,*}(y))^2)\in{\rm ch}(K^\ell)$ for all $y\in\R_+$ and
\begin{align*}
	&\Bigg(\frac{\sfY^\ell_t-\sfY^\ell_s}{t-s},\frac{[\sfX^\ell]_t-[\sfX^\ell]_s}{2(t-s)}\Bigg)=\frac{1}{t-s}\Bigg(\int_s^tb^{\ell,*}(\sfW_\theta)d\theta,\int_s^tq^{\ell,*}(\sfW_\theta) d\theta\Bigg)
\end{align*}
which belongs to ${\rm ch}(K^\ell)$ as an average of elements of the convex hull.

We now go back to the argument leading to \eqref{114}, and show that, with the system $\calS^*$ just constructed,
\eqref{114} holds with equality. To this end, note by the construction of $\sfX$ and $\sfY$,
and the definition of $b^*$ and $\sig^*$, that
\begin{align*}
	\int_0^{\tau}e^{-s}\Big(\frac{1}{2}f''(\sfW_s)d[\sfX^\ell]_s+f'(\sfW_s)d\sfY^\ell_s\Big)&=\int_0^{\tau}e^{-s}\brac{q^{\ell,*}(\sfW_s)f''\brac{\sfW_s}+b^{\ell,*}(\sfW_s)f'\brac{\sfW_s}}ds\\
	&=\int_0^{\tau}e^{-s}\max_{(b^\ell,q^\ell)\in\text{ch}(K^\ell)}\lt\{q^\ell f''(\sfW_s)+b^{\ell}f'\brac{\sfW_s}\rt\}ds.
\end{align*}
Reviewing the proof of \eqref{114} we see that it holds as equality in this case.
We finally take $t\to\infty$ as before, except that this time we apply the bounded convergence theorem for the second term,
appealing to \eqref{115}.
This proves the identity \eqref{eq:u a}.

\noi{\bf Step 3.}
By the dynamic programming principle, the function $f=f^{(x)}$
defined on $[0,x]$ and the function $\sfV$ are related as $\sfV(w)=f^{(x)}(w)$, $0\leq w\leq x$.
As this is true for all $x\in\R_+$, it follows that the function $\sfV$ is a classical solution to \eqref{eq:HJB} on $\R_+$,
and by Step 1 it satisfies the boundary conditions attached to it.
To complete the proof of part 1 of the Proposition, one must show that if $u$ is any classical
solution then $u=\sfV$. This can be shown by a calculation very similar to that carried out in Step 2;
we omit the details.

\noi{\bf Step 4.}
We next prove Part 2.
We  consider a system constructed exactly as $\calS^*$ from Step 2
except that the function $u$ is used instead of $f^{(x)}$,
and with a slight abuse of notation denote it again by $\calS^*$.
Then the relation of this system to the SDE \eqref{eq:111}
is as claimed, and it remains to prove that it is optimal for the MCP
(which is different than the statement already proved,
that it is optimal for the control problem \eqref{eq:u a}).
A calculation similar to that given in Step 2,
replacing $\tau$ by $t$ and using the boundary condition $\sfV'(0)=0$,
gives a version of \eqref{eq:bnd by max1} with equality, namely
\[
e^{-t}\EE[\sfV(\sfW_t)]=\sfV(w)-\EE\Big[\int_0^te^{-s}\sfW_sds\Big].
\]
Thus optimality follows once it is shown that the LHS above converges to $0$ as $t\to\iy$.
It is easy to see by \eqref{eq:111} that $\EE[\sfW_t]\le c(1+t)$ (where $c$ may depend on $w$
but not on $t$). Combined with \eqref{115} this shows $\EE[|\sfV(\sfW_t)|]\le c(1+t)$.
Thus the LHS above converges to $0$, and optimality is proved.
The statement relating $b$ and $\sig$ to $\bar{\mathbb{H}},\mathbb{H}^\ell$
follows by the construction in Step 2.
Specifically, this statement is exactly equations \eqref{736} and \eqref{eq:111}.

\noi{\bf Step 5.} For Part 3, note that nonnegativity of $\sfV$ follows from its definition.
Monotonicity follows by the monotonicity of $w\mapsto\Gam_1(w+\ph)$
for any trajectory $\ph$.
For convexity, we first show that
$w\mapsto\Gam_1(w+\ph)$ is convex for every $\ph$.
For any pair $\alpha,\alpha'\geq0$, $\alpha+\alpha'=1$, and any $w,w'\geq 0$,
using the formula $\Gam_1(\ph)(t)=\ph_t+\sup_{s\le t}(-\ph_s)^+$,
\begin{align*}
	\Gam_1(\alpha w+\alpha'w'+\ph)(t)&=
	\alpha w+\alpha'w'+\ph_t+\sup_{s\in[0,t]}(-\alpha w-\alpha'w'-\ph_s)^+\\
	&=\alpha w+\alpha'w'+\ph_t
	+\sup_{s\in[0,t]}\big(\alpha (-w-\ph_s)+\alpha'(-w'-\ph_s)\big)^+\\
	&\leq\alpha(w+\ph_t)+\alpha'(w'+\ph_t)
	+\alpha\sup_{s\in[0,t]}(-w-\ph_s)^++\alpha'\sup_{s\in[0,t]}(-w'-\ph_s)^+\\
	&=\alpha\Gam_1(w+\ph)(t)+\alpha'\Gam_1(w'+\ph)(t).
\end{align*}
Let $\calS\in\frS_{\alpha w+\alpha'w'}$ be given.
Then $\sfW=\Gam_1(\al w+\al'w'+\sfX+\sfY)$.
Construct $\calS^{(w)}$ and $\calS^{(w')}$ using the same probability space
by setting $\sfW^{(w)}=\Gam_1(w+\sfX+\sfY)$
and $\sfW^{(w')}=\Gam_1(w'+\sfX+\sfY)$.
Then $\sfW\le\al\sfW^{(w)}+\al'\sfW^{(w')}$.
Hence
\[
\sfC(\calS)\le\al\sfC(\calS^{(w)})+\al'\sfC(\calS^{(w')})
\le\alpha\sfV(w)+\alpha'\sfV(w'),
\]
and maximizing over $\calS\in\frS_{\alpha w+\alpha'w'}$ yields convexity.
\QED

\section{The upper bound}
\label{sec:ub}

The goal of this section and the next is to prove Theorem \ref{thm:Vn to V},
where this section is devoted to establishing
\begin{equation}\label{101}
\limsup_nV^n\le \sfV(0).
\end{equation}
To prove it, note that $V^n\le \sup_{a^n}C^n(\beta^{*n}(a^n),a^n)$.
Therefore
it suffices to show that for every sequence $a^n$ of admissible controls for the adversary,
one has
\begin{equation}\label{103}
\limsup_n C^n(\beta^{*n}(a^n),a^n)\le\sfV(0).
\end{equation}
To prove \eqref{103}, fix $a^n$.
Lemma \ref{lem:X tight} establishes
the $\calC$-tightness of the tuple $\{\hat{X}^{\ell n},[\hat{X}^{\ell n}],\hat{Y}^{\ell n},\hat{W}^{\tot,n},\hat{R}^n\}_{n\in\N}$.
Thus to prove \eqref{103}, it suffices to show that \eqref{103} holds
along any subsequence where the tuple converges. In what follows
we fix a subsequence where the tuple converges, and denote its limit
by $\{\sfX^{\ell},\sfU^\ell,\sfY^{\ell},\sfW^{\tot},\sfR\}$. (We later prove that $\sfU^\ell=[\sfX^{\ell }]$).
WLOG, the limit processes are defined on the original probability space.

The main step toward proving this upper bound will be to show that the limit tuple
defines an admissible system for the MCP, as we state below in Proposition \ref{pr:upper bnd}.
We first need to establish the following.

\begin{lemma}\label{lem:123}
	\begin{enumerate}
		\item \label{it: psi mart} Let $\psi$ be defined by
		$\psi^{\ell n}_t=\frac{1}{n}\sum_{k=1}^{nt}\big(\tfrac{1}{2}(\bar{J}_k^{\ell n})^2-\bq_{k}^{\ell n}\big)$,
		and let $\gamma^{\ell\ell' n}_t=[\hat{X}^{\ell n},\hat{X}^{\ell' n}]_t$.
		Then  $\psi^{\ell n}$ and $\gamma^{\ell \ell' n}$ are martingales w.r.t. $\{\calG_t^n\}$
		for $\ell\ne\ell'$.
		\item \label{it:psi to 0} As $n\to\iy$, $\psi^{\ell n}\to0$ and $\gamma^{\ell\ell' n}\to0$ in probability
		for $\ell\neq\ell'$.
	\end{enumerate}
\end{lemma}

\begin{proof}
Part 1.
	The proof is based on $(\bsigma_k^{\ell n})^2/2=\bq_k^{\ell n}$ and on the fact that for $m\leq k-1$
	\begin{align}\label{eq:L7.2-1}
		\E[(\bar{J}_k^{\ell n})^2-(\bsigma^{\ell n})^2_k|\calF_m^n]=\E\big[\E[(\bar{J}_k^{\ell n})^2-(\bsigma_k^{\ell n})^2|\calF^n_{k-1}]\big|\calF^n_m\big]=\E\big[\E[(\bar{J}_k^{\ell n})^2|\calF^n_{k-1}]-(\bsigma_k^{\ell n})^2\big|\calF^n_m\big]=0,
	\end{align}
	by the tower property and the definitions of $\bar{J}_k^{\ell n}$ and $(\bsigma_k^{\ell n})^2$.

	For $\psi^{\ell n}$, if $\floor{nt}=\floor{ns}$ for $s<t$, then $\E[\psi_t^{\ell n}|\calG^n_s]=\psi^{\ell n}_s$.
	Otherwise, by \eqref{eq:L7.2-1},
	\begin{align*}
		\E\lt[\psi^{\ell n}_t-\psi^{\ell n}_s\middle|\calG^n_s\rt]&=\frac{1}{2n}\sum_{k= \floor{ns}+1}^{nt}\E[(\bar{J}_k^{\ell n})^2-(\bsigma_{k}^{\ell n})^2|\calG^n_s]=0.
	\end{align*}
	For $\gamma^{\ell\ell' n}$,
	\begin{align*}
		\E\big[\gamma_t^{\ell\ell' n}-\gamma_s^{\ell\ell' n}\big|\calG^n_s\big]&=\frac{1}{n}\sum_{k=\floor{ns}+1}^{nt}\E\big[\bar{J}_k^{\ell n}\bar{J}_k^{\ell' n}\big|\calG^n_s\big]\\
		&=\frac{1}{n}\sum_{k=\floor{ns}+1}^{nt}\E\Big[\E\big[\bar{J}_k^{\ell n}\bar{J}_k^{\ell' n}\big|\calF_{k-1}^n\big]\Big|\calG^n_s\Big]\\
		&=\frac{1}{n}\sum_{k=\floor{ns}+1}^{nt}\E\Big[\E\big[\bar{J}_k^{\ell n}\big|\calF^n_{k-1}\big]\E\big[\bar{J}_k^{\ell' n}\big|\calF^n_{k-1}\big]\Big|\calG^n_s\Big]=0,
	\end{align*}
	where the second identity follows from the conditional independence of $J_k^{\ell n}$ and $J_k^{\ell' n}$ stated in \eqref{22}.
	
Part 2.
The main tool for estimating the two martingales is Burkholder's inequality.
	Starting with $\psi^{\ell n}$, by the identity $(\bsigma_k^{\ell n})^2/2=\bq_k^{\ell n}$,
	\begin{align*}
		\E\big[\|\psi^{\ell n}\|_{t_0}^2\big]\leq&  c\E\big[\big[\psi^{\ell n}\big]_{t_0}\big]\\
		=&\frac{c}{n^2}\E\bigg[\sum_{k=1}^{nt_0}((\bar{J}_k^{\ell n})^2-(\bsigma_{k}^{\ell n})^2)^2\bigg]\\
		=&\frac{c}{n^2}\sum_{k=1}^{nt_0}\E\big[\E[(\bar{J}_k^{\ell n})^4-2(\bar{J}_k^{\ell n})^2(\bsigma_k^{\ell n})^2+(\bsigma_{k}^{\ell n})^4|\calF_{k-1}^n]\big]\\
		\leq&\frac{c}{n^2}\sum_{k=1}^{nt_0}\E\big[(\bar{J}_k^{\ell n})^4\big]\\
		\leq&c\frac{nt_0}{n^2}\to 0,
	\end{align*}
	as $n\to\iy$,
	where we used \eqref{22} and the measurability of $\bsigma_k^2$ with respect to $\calF^n_{k-1}$, and the last inequality is due to Assumption \ref{ass:Pi}.\ref{it:4th moment}.
Finally for $\gamma^{\ell\ell' n}$, $\ell\neq\ell'$, by \eqref{22},
	\begin{align*}
		\E\big[\|\gamma^{\ell\ell' n}\|_{t_0}^2\big]&\leq c\E\big[[\gamma^{\ell\ell' n}]_{t_0}\big]=\frac{c}{n^2}\E\bigg[\sum_{k=1}^{nt_0}(\bar{J}_k^{\ell n})^2(\bar{J}_k^{\ell' n})^2\bigg]\le\frac{c nt_0}{n^2}\to 0,
	\end{align*}
	as $n\to\iy$.
\end{proof}

Recall that we have denoted the subsequential limit by $((\sfX^\ell,\sfY^\ell,\sfU^\ell)_\ell,\sfW^\tot,\sfR)$.
Denote
\[
\bar\calF_t=\sig\{\sfX^\ell_s,\sfY^\ell_s:s\in[0,t],\ell\in[L]\}.
\]

\begin{proposition}\label{pr:upper bnd}
The tuple $\calS=(\Om,\calF,\PP, \{\bar\calF_t\},(\sfX^\ell,\sfY^\ell),\sfW^\tot,\sfR)$
forms an admissible system for the MCP for $w=0$.
\end{proposition}

\begin{proof}
	Property 2 of the MCP follows directly from the definition of the filtration.
	
	For Property 3, recall \eqref{eq:W tot}.
	By the weak convergence of $(\hat X^{\ell n},\hat Y^{\ell n})_\ell$,
	along the subsequence that has been fixed,
	the continuity of the sample paths of $(\sfX^\ell,\sfY^\ell)_\ell$,
	and the convergence of $\hat e^n$ to zero, we have
	that
	$w+\hat X^{\tot,n}+\hat Y^{\tot,n}+\hat e^n\To\sfZ:=w+\sf X^\tot+\sfY^\tot$.
	Since $\Gam$ is a continuous mapping,
	we obtain $(\sfW^\tot,\sfR)=\Gam(\sfZ)$.
	
	For Property 4, first we shall prove the following statement:
	\begin{equation}\label{308}
	[\sfX^\ell,\sfX^{\ell'}]=0, \quad \ell\ne\ell',\quad \text{and} \quad [\sfX^\ell]=\sfU^\ell.
	\end{equation}
To this end we invoke \cite[Corollary VI.6.30]{jacod2013}.
To do this, note that
\begin{align*}
	\E[\|\Del \hat{X}^{\ell n}\|_t]^2 &\leq n^{-1}\E\bigg[\max_{1\leq k\leq \lt\lfloor nt\rt\rfloor}|\bar{J}_k^{\ell n}|\bigg]^2
\\
	&\le n^{-1}\E\bigg[\max_{1\leq k\leq \floornt}(\bar{J}_k^{\ell n})^2\bigg]\leq n^{-1}\E\bigg[\sum_{k=1}^{nt}(\bar{J}_k^{\ell n})^2\bigg]\leq ct,
\end{align*}
where $c$ does not depend on $n$.
Hence the assumptions of \cite[Corollary VI.6.30]{jacod2013}
are verified and it follows that
$(\hat{X}^n,[\hat{X}^n])\To(\sfX,[\sfX])$. Hence we obtain the identity $[\sfX^\ell]=\sfU^\ell$.
	By Lemma \ref{lem:123}.2, $[\hat X^{\ell n},\hat X^{\ell' n}]\to0$ in probability
	for $\ell\ne\ell'$. Therefore we have $[\sfX^\ell,\sfX^{\ell'}]=0$, $\ell\ne\ell'$.
	This shows \eqref{308}.
	
	The first part of Property 4 of the MCP follows.
	As for its second part, recall the definition of the processes $\bb_k^{\ell n}$ and $\bsigma_k^{\ell n}$ from \eqref{33} and the definition of $Y_t^{\ell n}$  in \eqref{32+}.
	Then
	\begin{align*}
		\Big(\frac{\hat{Y}^{\ell n}_t-\hat{Y}^{\ell n}_s}{t-s},\frac{[\hat{X}^{\ell n}]_t-[\hat{X}^{\ell n}]_s}{2(t-s)}\Big)-\frac{1}{nt-ns}\sum_{k=\floor{ns}+1}^{nt}(\bb_k^{\ell n},\bq_k^{\ell n})=\bigg(0,\frac{1}{t-s}\sum_{k=\floor{ns}+1}^{nt}\frac{\frac{1}{2}(\bar{J}_k^{\ell n})^2-\bq_k^{\ell n}}{n}\bigg).
	\end{align*}
	The RHS converges to 0 in probability by Lemma \ref{lem:123}.2.
	It follows from \eqref{308} that $(\hat{Y}^{\ell n},[\hat{X}^{\ell n}])\To(\sfY^\ell,[\sfX^\ell])$.
	Therefore the sequence $(nt-ns)^{-1}\sum_{k=\floor{ns}+1}^{nt}(\bb_k^{\ell n},\bq_k^{\ell n})$ has a limit,
	and it remains to show that this limit is in ${\rm ch}(K^\ell)$.
	To show this, note that $(\bb_k^{\ell n},\bq_k^{\ell n})\in K^{\ell n}$. Therefore, for $n>(t-s)^{-1}$,
	\begin{align}\label{34}
		\tfrac{1}{\floor{nt}-\floor{ns}-1}\sum_{k=\floor{ns}+1}^{nt}(\bb_k^{\ell n},\bq_k^{\ell n})\in\text{ch}(K^{\ell n}),
	\end{align}
	as a convex combination of members of $K^{\ell,n}$.
	We have assumed that $K^{\ell n}\to K^\ell$ in $d_{\rm H}$.
	By the definition of $d_{\rm H}$, it is elementary to show
	that this implies ${\rm ch}(K^{\ell n})\to {\rm ch}(K^\ell)$ in $d_{\rm H}$.
	It is also elementary that the limit of a convergent sequence of members of ${\rm ch}(K^{\ell n})$ must lie in ${\rm ch}(K^\ell)$. We omit the details. As a result, for every $s<t$,
	the limit of \eqref{34}, given by $(t-s)^{-1}(\sfY^\ell_t-\sfY^\ell_s,([\sfX^\ell]_t-[\sfX^\ell]_s)/2)$, must lie in ${\rm ch}(K^\ell)$, $\PP$-a.s. Finally, since the process is continuous, this claim holds true simultaneously for all $s<t$, $\PP$-a.s.
	
	Finally, to prove Property 1, i.e.\ that $\sfX$ is a $\bar\calF$-martingale,
	note that it is adapted and, by Fatou's lemma and Lemma \ref{lem: X is UI},
	$\E[|\sfX^\ell_t|]\leq\liminf_{n\to\infty}\E[|\hat{X}^{\ell n}_t|]<\infty$.
	It remains to show that for every $t\ge0$, $u>0$ and $A\in\bar\calF_t$,
	\begin{align}\label{b4}
		\E[\sfX^\ell_{t+u}\mathbbm{1}_A]=\E[\sfX_{t}^\ell\mathbbm{1}_A].
	\end{align}
	To this end we first show that for any $k\in\N$ and any continuous bounded function $g_k:\R^{2k}\to\R$,
	\begin{align}\label{eq:X mart}
		\E\big[\sfX^\ell_{t+u}g_k(\rho_{t_1\ldots t_k}(\sfX^\ell,\sfY^\ell))\big]=\E\big[\sfX^\ell_tg_k(\rho_{t_1\ldots t_k}(\sfX^\ell,\sfY^\ell))\big],
	\end{align}
	where $\rho_{t_1\ldots t_k}$ is the natural projection from $\calD^2[0,\infty)$ to $\R^{2k}$ at  $0\leq t_1\leq\cdots\leq t_k\leq t$.
	Indeed, by the martingale property of $\hat{X}^{\ell n}$,
	\begin{align*}
		\E\big[\hat{X}^{\ell n}_{t+u}g_k(\rho_{t_1\ldots t_k}(\hat{X}^{\ell n},\hat{Y}^{\ell n}))\big]=\E\big[\hat{X}^{\ell n}_{t}g_k(\rho_{t_1\ldots t_k}(\hat{X}^{\ell n},\hat{Y}^{\ell n}))\big].
	\end{align*}
	Hence \eqref{eq:X mart} follows by
	the weak convergence and uniform integrability of $\hat{X}^{\ell n}$.
Next we argue by the monotone class theorem.
	Define
	\begin{align*}
		\calK&=\{g_k\circ\rho_{t_1\ldots t_k}(\sfX^\ell,\sfY^\ell):\Om\to\R:g_k\text{ is bounded and continuous},k\in\N,0\leq t_1\leq...\leq t_k\leq t\},\\
		\calH&=\{f:\Om\to\R:f\text{ is $\bar\calF_t$-measurable, }\E[\sfX^\ell_{t+u}f]=\E[\sfX^\ell_{t}f]\}.
	\end{align*}
	Then $\calK$ is closed under multiplications, $\calH$ is a vector subspace, and
	$\calK\subset\calH$ according to \eqref{eq:X mart}.
Moreover, the constant function $1$ is in $\calH$ since $\E[\sfX^\ell_{t+u}]=\E[\sfX^\ell_t]$
by \eqref{eq:X mart}.
Finally,
the intersection of $\calH$ with the set of non-negative functions is closed under bounded increasing limits, by the monotone convergence.
Therefore $\calH$ contains all bounded functions that are measurable w.r.t.\ $\sigma(\calK)$, but $\sigma(\calK)=\bar \calF_t$.This shows \eqref{b4} and completes the proof.
\end{proof}

We can now complete the proof of \eqref{103}.
By Proposition \ref{pr:upper bnd}, the tuple $\calS$ satisfies
$\calS\in\frS_0$. Hence
$\sfC(\calS)\leq\sup_{\calS'\in\frS_0}\sfC(\calS')=\sfV(0)$.
Thus the proof of \eqref{103} will be complete once it is shown that,
along the subsequence that has been fixed, one has
$\lim_n C(\beta^{*n}(a^n),a^n)=\sfC(\calS)$.
To this end, note that by assumption, $\hat W^{\tot,n}\To\sfW^\tot$
along this subsequence. Hence, along the same subsequence,
by Lemma \ref{lem:W To 0},
\[
	(\hat W^{1n},\hat W^{2n},...,\hat W^{Ln})\To\brac{\sfW^\tot,0,...,0},
\]
and by Lemma \ref{lem:Q-W to 0},
\[
	(\hat Q^{1n},\hat Q^{2n},...,\hat Q^{Ln})\To\brac{\mu^1\sfW^\tot,0,...,0}.
\]
Hence by the normalization $h^1\mu^1=1$, we have
$h\cdot\hat Q^n\To\sfW^\tot$.
To deduce from this convergence of the expected integrals, we appeal to uniform
integrability.
Consider the probability measure $d\PP\times e^{-t}dt$ on the product space
$(\Om\times\R_+,\calF\otimes\calR_+)$. W.r.t.\ this measure,
$h\cdot\hat Q^n$ are uniformly integrable as RVs on
the product space, as follows directly from Lemma \ref{lem:W,Q UI}.
As a result, along the subsequence,
\begin{align}\label{756}
\lim_n C(\beta^{*n}(a^n),a^n)=
	\lim_{n\to\infty}\expect{\int_0^\infty e^{-t}h\cdot\hat Q_t^ndt}
	=\expect{\int_0^\infty e^{-t}\sfW_t^\tot dt}=\sfC(\calS).
\end{align}
This completes the proof of \eqref{103}. The upper bound \eqref{101} follows.
\qed

\section{The lower bound}
\label{sec:ao}

The objective here is to prove
\begin{align}\label{eq:LB}
	\uu{V}:=\liminf_{n\to\infty}V^n\geq\sfV(0)
\end{align}
and then complete the proof of Theorem \ref{thm:Vn to V}.
To this end, we construct a sequence $\{a^{*n}\}$ that satisfies
\begin{equation}\label{58}
\sfV(0)\leq\liminf_{n\to\infty}C(\beta^{*n}(a^{*n}),a^{*n}).
\end{equation}
This suffices in order to prove that $\lim_nV^n=\sfV(0)$
because by Theorem \ref{thm:cmu is opt},
\begin{align*}
	\liminf_{n\to\infty}C(\beta^{*n}(a^{*n}),a^{*n})=\liminf_{n\to\infty}\inf_{\beta\in\B^n}C(\beta(a^{*n}),a^{*n})\leq\liminf_{n\to\infty}\inf_{\beta\in\B^n}\sup_{a\in\calA^n}C(\beta(a),a)=\uu{V}.
\end{align*}
Toward showing \eqref{58}, recall that
by Lemmas \ref{lem:Q-W to 0} and \ref{lem:W To 0}
and the choice $h^1\mu^1=1$,
when the SC uses strategy $\beta^{*n}$, one has
that $\hat Q^{\ell n}\to0$
for $2\leq \ell\leq L$ and $h\cdot\hat Q^n-\hat W^{\tot,n}\to0$
in probability.
Using this along with
a uniform integrability argument as in \eqref{756} yields that,
under $\{\beta^{*n}\}$ and any control sequence $\{a^{n}\}$,
\[
\liminf_{n\to\infty}C(\beta^{*n}(a^{n}),a^{n})=
\liminf_{n\to\infty}\expect{\int_0^\infty e^{-t}h\cdot \hat{Q}^n_tdt}=\liminf_{n\to\infty}\expect{\int_0^\infty e^{-t}\hat{W}^{\tot,n}_tdt}.
\]
Therefore, in order to prove \eqref{58} (hence \eqref{eq:LB}) it suffices to find a control sequence $\{a^{*n}\}$ such that under $(a^{*n},\beta^{*n})$
one has
\begin{equation}\label{a5}
\uu{U}:=\liminf_{n\to\infty}\expect{\int_0^\infty e^{-t}\hat{W}^{\tot,n}_tdt}\geq\sfV(0).
\end{equation}
We will now show that \eqref{a5} holds
under a suitable $\{a^{*n}\}$; in particular, we will assume
throughout what follows that the strategy $\beta^{*n}$ is used.

Again, in this proof, the label `tot' is suppressed.
Recall the expression given in \eqref{eq:W tot} for $(\hat W^n,\hat R^n)$.
Here it is more convenient to work with pure jump processes and
define  $(\tilde{W}^n,\tilde{R}^n)=\Gam(\hat{X}^n+\hat{Y}^n)$.
Because of the Lipschitz property of $\Gam$,
we have only a small error between
$(\tilde{W}^n,\tilde{R}^n)$ and $(\hat W^n,\hat R^n)$.
Specifically, for all $t$,
\begin{equation}\label{a6}
\|\tilde{W}^n-\hat W^n\|_t+\|\tilde{R}^n-\hat R^n\|_t
\le c\|\hat e^n\|_t\le cn^{-1/2}.
\end{equation}
Thus $\tilde{W}^n,\tilde{R}^n,\hat{X}^n,\hat{Y}^n$
are pure jump processes starting at $0$,
and $\tilde{W}^n=\hat{X}^n+\hat{Y}^n+\tilde{R}^n$.
Moreover, the jump sizes are
given by e.g.,
$\Del\tilde{W}_{k/n}^n=\tilde{W}^n_{k/n}-\tilde{W}^n_{(k-1)/n}$, $k\ge1$.
Denote $\ttW_k^n=\tilde{W}_{k/n}$ and $\ttR_k^n=\tilde{R}_{k/n}^n$.
We shall use the fact that, by the definition of $(\tilde W^n,\tilde R^n)$,
(and with $\bar J_k^{\tot,n}$,
$\bb_k^{\tot,n}$ abbreviated as $\bar J_k^n$, $\bb_k^n$),
\begin{equation}\label{a2}
\Del\ttW_k^n=\Del\hat X^n_{k/n}+\Del\hat Y^n_{k/n}+\Del\ttR^n_k
=n^{-1/2}\bar J_k^n+n^{-1}\bb_k^n+\Del\ttR^n_k.
\end{equation}

To construct the sequence $a^{*n}$, we invoke again a measurable selection
argument.
Let $\ph=\ph^{\ell n}:\R\times\PI^{\ell,n}\to\R$ be defined by
\begin{equation}\label{b1}
\ph(w,\pi)=u'(w)b^{\pi\ell n}+u''(w)q^\pi, \qquad w\in\R_+.
\end{equation}
Using Assumption \ref{ass:Pi}.1, this function is continuous in both variables, and since
$\PI^{\ell,n}$ is compact, $\ph^{(m)}:=\ph|_{[0,m]\times\PI^{\ell n}}$ is bounded for any $m\in\N$.
Hence by
\cite[Lemma 8.10]{budhiraja19}, there exists for each $m$ a measurable function
$\psi^{(m)}:[0,m]\to\PI^{\ell n}$ such that if $w\in[0,m]$ and $\pi=\psi^{(m)}(w)$,
\[
\ph^{(m)}(w,\pi)=\sup\{\ph^{(m)}(w,\pi'):\pi'\in\PI^{\ell n}\}.
\]
Next, a single measurable function $\psi=\psi^{\ell n}:\R_+\to\PI^{\ell n}$
can be extracted by letting
$\psi(w,\pi)=\psi^{(m)}(w,\pi)$ when $w\in[m-1,m)$. This way, we have similarly that
if $w\in\R_+$ and $\pi=\psi(w)$,
\begin{equation}\label{b2}
\ph(w,\pi)=\sup\{\ph(w,\pi'):\pi'\in\PI^{\ell n}\}.
\end{equation}

Using this measurable function, we define, for every $n,\ell$,
\begin{equation}\label{b3}
\pi^{\ell n}_k=\psi^{\ell n}(\ttW^n_{k-1}),\qquad k\ge1.
\end{equation}
This gives
\begin{equation}\label{eq:prelimit policy}
u'(\ttW_{k-1}^n)b^{\pi^{\ell n}_k,\ell n}+u''(\ttW^n_{k-1})q^{\pi^{\ell n}_k}
=
\max\{u'(\ttW^n_{k-1})b^{\pi\ell n}+u''(\ttW^n_{k-1})q^{\pi}:\pi\in\PI^{\ell n}\}.
\end{equation}
Once $\pi^{\ell n}_k$ is selected as above for all $\ell$, the tuple $(J^{\ell n}_k)_{\ell\in[L]}$
is selected according to \eqref{22}, and this
defines a control process that we shall denote $a^{*n}$.
The proof of the first part of the theorem
will be complete once we show that this control satisfies \eqref{58}.

To this end, fix $t>0$ and $x>0$ and define the $\{\calG_s^n\}$-stopping times
\begin{align*}
    \tilde\tau=\tilde\tau(x,n)=\inf\{s>0:\tilde{W}^n_s>x\},
    \quad \tau=\tau(x,t,n)=\tfrac{\floornt}{n}\land\tilde\tau(x,n).
\end{align*}
Since $\tilde W^n$ is piecewise constant with jumps at times that are multiples
of $n^{-1}$, $n\tau$ takes values in $\N\cup\{\iy\}$.
In addition, $n\tau$ is an $\{\calF_k^n\}$ stopping time since
$\{n\tau\leq k\}=\{\tau\leq\tfrac{k}{n}\}\in\calG^n_{k/n}=\calF^n_k$.
Applying Taylor expansion with the Lagrange remainder, there exists $\zeta_k^n$ taking a value between $\ttW^n_k$ and $\ttW^n_{k-1}$ such that 
\begin{align*}
    e^{-\tau}u(\tilde{W}^n_{\tau})&=u(0)-\int_0^{\tau}e^{-s}u(\tilde{W}^n_{s-})ds+\sum_{k=1}^{ n\tau}e^{-\frac{k}{n}}\brac{u\brac{\ttW^n_{k}}-u\brac{\ttW^n_{k-1}}}\nonumber\\
    &=u(0)-\int_0^{\tau}e^{-s}u(\tilde{W}^n_{s-})ds+\sum_{k=1}^{ n\tau}e^{-\frac{k}{n}}\Big(\brac{\ttW^n_k-\ttW^n_{k-1}}u'\brac{\ttW^n_{k-1}}+\frac{1}{2}\brac{\ttW^n_{k}-\ttW^n_{k-1}}^2u''\brac{\zeta_k^n}\Big)\\
    &=u(0)-\int_0^{\tau}e^{-s}u(\tilde{W}^n_{s-})ds\\
    &\quad+\sum_{k=1}^{ n\tau}e^{-\frac{k}{n}}\Big(\brac{\ttW^n_{k}-\ttW^n_{k-1}}u'\brac{\ttW^n_{k-1}}+\frac{1}{2}\brac{\ttW^n_{k}-\ttW^n_{k-1}}^2u''\brac{\ttW^n_{k-1}}\Big)+\eps^n_1(\tau)
\end{align*}
with
\begin{align*}
    \eps_1^n=\frac{1}{2}\sum_{k=1}^{n\tau}e^{-k/n}\brac{u''\brac{\zeta_k^n}-u''\brac{\ttW^n_{k-1}}}\brac{\ttW^n_{k-1}-\ttW^n_{k}}^2.
\end{align*}
Using \eqref{a2},
\begin{align}\label{eq:prelimit Taylor}
    e^{-\tau}u(\tilde{W}^n_{\tau})=u(0)-\int_0^{\tau}e^{-s}u(\tilde{W}^n_{s-})ds+\sum_{k=1}^{ n\tau}\frac{1}{n}e^{-\frac{k}{n}}T_k^n+\sum_{k=1}^3\eps^n_i
\end{align}
with
\begin{align*}
&T^n_k=\bb^n_ku'\brac{\ttW^n_{k-1}}+\frac{1}{2}(\bar{J}^n_k)^2u''\brac{\ttW^n_{k-1}},
\\
    &\eps^n_2=\sum_{k=1}^{n\tau}e^{-\frac{k}{n}}\brac{\frac{\bar{J}^n_k}{\sqrt{n}}\rho_k^n+\frac{1}{2}(\rho^n_k)^2}u''\brac{\ttW^n_{k-1}},
    &\rho_k^n=\frac{\bb^n_k}{n}+\Del\ttR^n_k,\\
    &\eps^n_3=\sum_{k=1}^{n\tau}e^{-\frac{k}{n}}\brac{\frac{\bar{J}^n_k}{\sqrt{n}}+\Del\ttR^n_k}u'\brac{\ttW^n_{k-1}}.
\end{align*}

Consider the $k$th term in the first sum in \eqref{eq:prelimit Taylor}.
Denote $H_k=\indicator{k-1<n\tau}$.
Note that $\bb^n_k$ and $\ttW^n_{k-1}$ are $\calF^n_{k-1}$-measurable and that $\E[(\bar{J}^n_k)^2|\calF^n_{k-1}]=\bq_k^{\tot,n}=:\bq_k$ due to the independence of $(J^{\ell n}_k)_\ell$ conditioned on $\calF^n_{k-1}$.
Hence by \eqref{eq:prelimit policy} and \eqref{eq:HJB},
\begin{align*}
    \expect{T^n_kH_k}
    &=\expect{\brac{\bb^n_ku'\brac{\ttW^n_{k-1}}+\frac{1}{2}\E\big[(\bar{J}_k^n)^2\big|\calF^n_{k-1}\big]u''\brac{\ttW^n_{k-1}}}H_k}\\
    &=\expect{\brac{\bb^n_ku'\brac{\ttW^n_{k-1}}+\bq^n_ku''\brac{\ttW^n_{k-1}}}H_k}\\
    &=\E\bigg[\sum_{\ell=1}^L\brac{\bb_k^{\ell n}u'\brac{\ttW^n_{k-1}}+\bq_k^{\ell n} u''\brac{\ttW^n_{k-1}}}H_k\bigg]\\
    &=\E\bigg[\sum_{\ell=1}^L\max_{\pi\in\PI^\ell}\lt\{b^{\pi}u'\brac{\ttW^n_{k-1}}+q^\pi u''\brac{\ttW^n_{k-1}}\rt\}H_k\bigg]\\
    &=\E\bigg[\sum_{\ell=1}^L\max_{(b^\ell,q^\ell)\in K^{\ell n}}\lt\{b^{\ell}u'\brac{\ttW^n_{k-1}}+q^{\ell} u''\brac{\ttW^n_{k-1}}\rt\}H_k\bigg].
\end{align*}
Toward using the HJB equation we need to address the discrepancy between the maximization over $K^{\ell n}$ and over $K^\ell$.
To this end we use the following fact, which follows directly from the assumption that $K^{\ell n}\to K^\ell$ in $d_{\rm H}$ and the compactness of $K^{\ell n}$ and $K^\ell$.
Namely, for any compact $C\subset\R^2$,
	\begin{align*}
		\lim_n\sup_{\vc{x}\in C}|\max_{\vc{y}\in K^{\ell n}}\vc{y}\cdot\vc{x}-\max_{\mathbf{y}\in K^\ell}\mathbf{y}\cdot\mathbf{x}|=0.
	\end{align*}
Moreover, by \cite[Theorem 32.2]{rockafellar1970}
$\max_{\mathbf{y}\in K^\ell}\mathbf{y}\cdot\mathbf{x}=
\max_{\mathbf{y}\in {\rm ch}(K^\ell)}\mathbf{y}\cdot\mathbf{x}$.
We apply these two facts with
\[
C=[0,\|u'\|_x]\times[0,\|u''\|_x],
\]
$\vc{x}=(u'(\ttW^n_{k-1}),u''(\ttW^n_{k-1}))$ and $\vc{y}=(b^\ell,q^\ell)$.
Then for all $\ell\in[L]$
\begin{align*}
	\max_{(b^\ell,q^\ell)\in K^{\ell n}}\lt\{b^{\ell}u'\brac{\ttW^n_{k-1}}+q^{\ell} u''\brac{\ttW^n_{k-1}}\rt\}=\max_{(b^\ell,q^\ell)\in {\rm ch}(K^{\ell})}\lt\{b^{\ell}u'\brac{\ttW^n_{k-1}}+q^{\ell} u''\brac{\ttW^n_{k-1}}\rt\}+\tilde \eps_{4,k}^{\ell n},
\end{align*}
where $|\tilde \eps_{4,k}^{\ell n}|\le\tilde \eps^n_4$ and $\tilde \eps_4^n$ is a deterministic sequence
(which depends on $x$ but not on $k$ or $\ell$) converging to zero as $n\to\iy$.
The above, together with the fact that $u$ solves \eqref{eq:HJB}, results in 
\begin{align*}
	\E\Big[\sum_{k=1}^{ n\tau}\frac{1}{n}e^{-\frac{k}{n}}T_k^n\Big]&=\E\Big[\sum_{k=1}^\iy\frac{e^{-k/n}}{n}T^n_kH_k\Big]\\
	&=\sum_{k=1}^\iy\frac{e^{-k/n}}{n}\E\bigg[\bigg(\sum_{\ell=1}^L\max_{(b^\ell,q^\ell)\in {\rm ch}(K^{\ell})}\lt\{b^{\ell}u'\brac{\ttW^n_{k-1}}+q^{\ell n} u''\brac{\ttW^n_{k-1}}\rt\}+\tilde \eps_{4,k}^{\ell n}\bigg) H_k\bigg]\\
	&=\E\Big[\sum_{k=1}^\iy\frac{e^{-k/n}}{n}\brac{u\brac{\ttW_{k-1}}-\ttW_{k-1}}H_k\Big]+\E[\eps_{4}]\\
	&=\E\Big[\sum_{k=1}^{n\tau}\frac{e^{-k/n}}{n}\brac{u\brac{\ttW_{k-1}}-\ttW_{k-1}}\Big]+\E[\eps_4],
\end{align*}
with $\eps_4^n=\sum_{k=1}^{n\tau}n^{-1}e^{-k/n}\tilde\eps_{4k}^{\tot,n}$.
Using the bound on $\tilde\eps^{\ell n}_{4,k}$,
\[
|\eps_4^n|\le n^{-1}ntL\tilde\eps_4^n=Lt\tilde\eps^n_4.
\]
This shows that
$\E[|\eps^n_4|]\leq tL\tilde \eps^n_4\to 0$ as $n\to\iy$.
Taking expectation in \eqref{eq:prelimit Taylor} gives
\begin{align}
\notag
    \E\big[e^{-\tau}u(\tilde{W}^n_\tau)\big]&=u(0)+\E\bigg[-\int_0^{\tau}e^{-s}u(\tilde{W}^n_{s})ds+\sum_{k=1}^{n\tau}\frac{1}{n}e^{-\frac{k}{n}}\brac{u\brac{\ttW^n_{k-1}}-\ttW^n_{k-1}}+\sum_{i=1}^4\eps^n_i\bigg]
    \\
    &=
      u(0)+\E\bigg[-\int_0^{\tau}e^{-s}\tilde{W}^n_sds+\sum_{i=1}^6\eps^n_i\bigg],
\label{eq:prelimit Taylor 2}
\end{align}
where
\begin{align*}
    \eps_5^n=&\frac{1}{n}\sum_{k=1}^{ n\tau}e^{-k/n}u\brac{\ttW^n_{k-1}}-\int_0^{\tau}e^{-s}u(\tilde{W}^n_{s})ds,\\
    \eps_6^n=&-\frac{1}{n}\sum_{k=1}^{ n\tau}e^{-k/n}\ttW^n_{k-1}+\int_0^{\tau}e^{-s}\tilde{W}^n_{s}ds.
\end{align*}
The next step is to show that $\E[\eps^n_i]\to0$ for $1\le i\le 6$, where
the case $i=4$ has already been addressed.

For $\eps_2^n$, let $\calK_1=\calK_1^n=\{k\in[nt]:\Del\ttR^n_k>0\}$ and sum
separately over $k\in\calK_1$ and over $k\in\calK_1^c$.
For $k\in\calK_1^c$, we have $|\rho^n_k|\leq c/n$.
For $k\in\calK_1$, recall the definition of $\bxi^n$ and observe that since $\Del\ttR^n_k>0$ only if $\ttW^n_k=0$, one has by \eqref{a2},
\begin{align}\label{eq:DR>0}
    0<\Del\ttR^n_{k}=-\frac{1}{\sqrt{n}}\bar{J}^n_k-\frac{\bb^n_k}{n}-\ttW^n_{k-1}\leq-\frac{J^n_k-\bxi^n_k}{\sqrt{n}}-\frac{\bb^n_k}{n}=-\frac{J^n_k}{\sqrt{n}}+\frac{1}{\sqrt{n}},
\end{align}
where the last equality follows from the relation between $\bxi^n_k$ and $\bb^n_k$
expressed in Assumption \ref{ass:Pi}(2) and Assumption \ref{assn:ht} by which
$1/\mu^\ell$ sum to 1.
Therefore, in these times,
\begin{equation}\label{a3}
\text{$0\le J^n_k\leq 1$ and $\Del\ttR^n_k\leq n^{-1/2}$}.
\end{equation}
This implies
\begin{align*}
    |\rho_k^n|=\Big|\Del\ttR^n_k+\frac{\bb^n_k}{n}\Big|\leq\frac{2}{\sqrt{n}},
\end{align*}
for large $n$.
Note that \eqref{eq:DR>0} also implies $\ttW^n_{k-1}\leq 2n^{-1/2}$ for $k\in\calK_1$.
From these bounds we obtain 
\begin{align*}
    &\Big|\frac{\bar{J}^n_k}{\sqrt{n}}+\frac{1}{2}\rho^n_k\Big|\leq\frac{3}{\sqrt{n}},\quad k\in\calK_1.
\end{align*}
Using these bounds in the definition of $\eps_2^n$,
\begin{align*}
    |\eps^n_2|=&\abs{\sum_{k=1}^{n\tau}e^{-\frac{k}{n}}\brac{\frac{\bar{J}^n_k}{\sqrt{n}}\rho^n_k+\frac{1}{2}(\rho^n_k)^2}u''\brac{\ttW^n_{k-1}}}\leq \norm{u''}_x\bigg(\sum_{k\in\calK_1}\frac{3}{\sqrt{n}}\abs{\rho^n_k}+\frac{c}{n}\sum_{k\in\calK_1^c}\bigg(\frac{\abs{\bar{J}^n_k}}{\sqrt{n}}+\frac{c}{n}\bigg)\bigg).
\end{align*}
Taking expectation in the sum over $\calK_1^c$,
\begin{align*}
    \E\bigg[\frac{c}{n}\sum_{k\in\calK_1^c}\brac{\frac{\abs{J_k^n-\bxi^n_k}}{\sqrt{n}}+\frac{c}{n}}\bigg]&\leq\frac{c}{n}\sum_{k=1}^{nt}\frac{\expect{\abs{J^n_k-\bxi^n_k}}}{\sqrt{n}}+\frac{c}{n^2}nt\leq\brac{\frac{c}{n^{3/2}}+\frac{c}{n^2}}nt,
\end{align*}
which converges to zero with $n\to\infty$.
Next, for the sum over $\calK_1$, bound $|\rho_k^n|$ by $c/n+\Del\ttR^n_k$ and
proceed with
\begin{align*}
    &\E\bigg[\sum_{k\in\calK_1}\frac{3}{\sqrt{n}}\brac{\frac{c}{n}+\Del\ttR^n_k}\bigg]\\
    &\leq\sum_{k=1}^{nt}\frac{c}{n^{3/2}}+3\expect{\sum_{k=1}^{nt}\brac{\frac{\Del\ttR^n_{k}}{\sqrt{n}}\indicator{0<\Del\ttR^n_{k}<n^{-3/4}}+\frac{\Del\ttR^n_{k}}{\sqrt{n}}\indicator{n^{-3/4}<\Del\ttR^n_{k}}}}\\
    &\leq nt\frac{c}{n^{3/2}}+3\expect{\sum_{k=1}^{nt}\brac{\frac{n^{-3/4}}{\sqrt{n}}\indicator{0<\Del\ttR^n_{k}<n^{-3/4}}+\frac{2}{n}\indicator{n^{-3/4}<\Del\ttR^n_{k}}}}\\
    &\leq t\frac{c}{\sqrt{n}}+3n^{-1/4}t+3\frac{2}{n}\expect{\abs{\calK_2}},
\end{align*}
where $\calK_2=\calK_2^n=\lt\{k\in[nt]:n^{-3/4}<\Del\ttR^n_{k}\rt\}$.
Now,
\begin{align*}
    n^{-1}\expect{\abs{\calK_2}}\leq n^{-1}\E\bigg[\sum_{k\in\calK_2}\frac{\Del\ttR^n_{k}}{n^{-3/4}}\bigg]\le n^{-1/4}\E[\tilde{R}^n_t].
\end{align*}
Using Lemma \ref{lem: X is UI} and \eqref{eq:bnd on Y} and the Lipschitz property of the Skorokhod map, we bound this further by
\begin{align}\label{a4}
    \E[\tilde{R}^n_t]&\leq 2\E[\Vert\hat{X}^n\Vert_t+\Vert\hat{Y}^n\Vert_t]\leq c(\sqrt{t}+t).
\end{align}
This shows $\E\lt[\eps^n_2\rt]\to 0$ as $n\to\infty$.

We next treat $\eps_1^n$. By \eqref{a3}, $\Del\ttR^n_k\le n^{-1/2}$, and
Thus by \eqref{a2}
 $|\Del\ttW^n_k|\leq|\bar{J}^n_k|/\sqrt{n}+c/\sqrt{n}$.
\begin{align*}
    \abs{\eps_1^n}&\leq \frac{c}{n}\sum_{k=1}^{n\tau}\abs{u''\brac{\zeta^n_k}-u''\brac{\ttW^n_{k-1}}}U_k^n,
\end{align*}
where $U_k^n=\brac{|\bar J^n_k|+1}^2$.
Note that $|\zeta^n_k-\ttW^n_{k-1}|\le|\Del\ttW^n_k|$.
Letting $\delta^n=n^{1/3}$, we have
\begin{align}
|\eps^n_1|   &\le\frac{c}{n}\sum_{k=1}^{n\tau}\brac{w_x\brac{u'',|\Del\ttW^n_k|}\indicator{|\bar{J}^n_k|\leq \delta^n}+2\|u''\|_x\indicator{|\bar{J}^n_k|> \delta^n}}U^n_k\nonumber\\
    &\leq\frac{c}{n}\sum_{k=1}^{n\tau}w_x\brac{u'',\frac{\abs{\bar{J}^n_k}+c}{\sqrt{n}}}U^n_k\indicator{\abs{\bar{J}^n_k}\leq \delta^n}+\norm{u''}_x\frac{c}{n}\sum_{k=1}^{n\tau}U^n_k\indicator{\abs{\bar{J}^n_k}> \delta^n}\nonumber\\
    &\leq w_x\brac{u'',\frac{2\delta^n}{\sqrt{n}}}\frac{c}{n}\sum_{k=1}^{nt}U^n_k\label{eq:err bound1}\\
    &\quad+\norm{u''}_x\frac{c}{n}\sum_{k=1}^{nt}U^n_k\indicator{\abs{\bar{J}^n_k}> \delta^n}\label{eq:err bound2}.
\end{align}
By the continuity of $u''$ and the bound $\E[U_k^n]\le c$,
the expectation of the expression in \eqref{eq:err bound1} is bounded by
$ctw_x\brac{u'',2n^{-1/6}}$, which converges to $0$ as $n\to\iy$.
By Assumption \ref{ass:Pi}.\ref{it:4th moment}, $\E[(U_k^n)^2]\le c$.
Hence the expectation of the expression in \eqref{eq:err bound2} is bounded by
\begin{align*}
\frac{c}{n}\sum_{k=1}^{nt}\sqrt{\E[(U_k^n)^2]\prob{|\bar{J}^n_k|>\delta^n}}
\leq ct\max_{k\in[nt]}\sqrt{\frac{\E[(\bar{J}^n_k)^4]}{n^{4/3}}}\to 0,
\end{align*}
where once again we have used Assumption \ref{ass:Pi}.\ref{it:4th moment}.
This completes the proof that $\E[\eps^n_1]\to0$.

For $\eps^n_6$, 
recalling that $\tilde{W}^n$ is a pure jump process, we have
$\int_0^{\tau}e^{-s}\tilde{W}^n_{s}ds=\sum_{k=1}^{ n\tau}\ttW^n_{k-1}\psi_k^n$,
where we denote $\psi_k^n=\int_{(k-1)/n}^{k/n}e^{-s}ds$.
    Now,
    \begin{align*}
        -\eps_6=\frac{1}{n}\sum_{k=1}^{ n\tau}e^{-k/n}\ttW^n_{k-1}-\sum_{k=1}^{ n\tau}\ttW^n_{k-1}\psi_k^n&=\sum_{k=1}^{ n\tau}\ttW^n_{k-1}\brac{\frac{e^{-k/n}}{n}-\psi_k^n},
    \end{align*}
    while $0\le \psi_k^n-n^{-1}e^{-k/n}\leq n^{-1}(1-e^{-1/n})$.
    Therefore $\abs{\eps^n_6}\leq t\brac{1-e^{-1/n}}\max_{0\leq k\leq n\tau-1}\ttW^n_{k}\le xt(1-e^{-1/n})$, where we used
$\ttW^n_k\leq x$ for $k\in[n\tau-1]$.
Hence, $\E[\eps^n_6]\to 0$ as $n\to\iy$.
A similar argument holds for $\eps^n_5$.

We are left with $\eps^n_3$.
Using \eqref{eq:DR>0}, we have that
whenever $\Del\ttR^n_{k}>0$, $\ttW^n_{k-1}$ is bounded above
by $n^{-1/2}$.
Hence, by \eqref{a4},
\begin{align*}
    \bigg|\E\bigg[\sum_{k=1}^{n\tau}e^{-\frac{k}{n}}u'\brac{\ttW^n_{k-1}}\Del\ttR^n_{k}\bigg]\bigg|&\leq\norm{u'}_{n^{-1/2}}\E[\tilde{R}^n_t]\leq c\norm{u'}_{n^{-1/2}}\big(\sqrt{t}+t\big)\to 0,
\end{align*}
by continuity of $u'$ and $u'(0)=0$.
Moreover, using the fact that $n\tau$ is an $\{\calF^n_k\}$-stopping time and $\E[\bar J^n_k|\calF^n_{k-1}]=0$, 
\begin{align*}
    \E\bigg[u'\brac{\ttW^n_{k-1}}\indicator{k<n\tau+1}\bar{J}^n_k\bigg]&=\E\bigg[u'\brac{\ttW^n_{k-1}}\indicator{k<n\tau+1}\E\big[\bar{J}^n_k\big|\calF^n_{k-1}\big]\bigg]=0.
\end{align*}
The last two displays imply $\E[\eps^n_3]\to 0$ as $n\to\infty$.

Having shown that $\E[\eps^n_i]\to0$, $1\le i\le 6$, we have by
\eqref{eq:prelimit Taylor 2} that
\[
\liminf_{n\to\infty}\expect{\int_0^{\iy}e^{-s}\tilde{W}^n_sds}
\ge u(0)-\limsup_{n\to\infty}\expect{e^{-\tau}u(\tilde{W}^n_\tau)}.
\]
Hence by the definition of $\uu{U}$ in \eqref{a5} and the bound \eqref{a6},
we have
\begin{equation}
\label{eq:prelimit Taylor 3}
\uu{U}
\ge u(0)-\limsup_{n\to\infty}\expect{e^{-\tau}u(\tilde{W}^n_\tau)}.
\end{equation}
Now,
\begin{align*}
    |\E[e^{-\tau}u(\tilde{W}^n_{\tau})]|
    \leq\sqrt{\expect{e^{-2\tau}}\E\big[u^2(\tilde{W}^n_{\tau})\big]}.
\end{align*}
Moreover, denoting $\eta=\eta(x,t,n)=\{\tilde\tau>\tfrac{\floornt}{n}\}$
and recalling $\tau=\tilde\tau\w\tfrac{\floornt}{n}$,
\begin{align}\label{eq:bound e-2tau}
    \expect{e^{-2\tau}}&=\expect{e^{-2\tau}\mathbbm{1}_\eta}+\expect{e^{-2\tau}\mathbbm{1}_{\eta^c}}\nonumber\\
    &\leq e^{-2(t-1)}+\PP(\eta^c)\nonumber\\
    &= e^{-2(t-1)}+\PP(\Vert\tilde{W}^n\Vert_t\geq x). \notag
\end{align}
By Lemma \ref{lem:W,Q UI},
$\PP(\Vert\tilde{W}^n\Vert_t\geq x)\leq x^{-1}\E[\Vert\tilde{W}^n\Vert_t]
\leq cx^{-1}(t+1)$.
Due to the linear growth of $u$ and Lemma \ref{lem:W,Q UI},
\begin{align*}
    \E\big[u^2(\tilde{W}^n_{\tau})\big]\leq\E\big[\Vert u^2\circ\tilde{W}^n\Vert_t\big]\leq c\big(t+1\big)^2.
\end{align*}
Hence by \eqref{eq:prelimit Taylor 3},
\[
\uu{U} \ge u(0)-c(e^{-2(t-1)}+x^{-1}(t+1))^{1/2}(t+1).
\]
We can now take $x\to\iy$ and then $t\to\iy$ to obtain
$\uu{U}\ge u(0)=\sfV(0)$.
This completes the proof of the lower bound \eqref{a5}.
In view of the results of the previous section, this establishes part 1
of Theorem \ref{thm:Vn to V}, namely that $\lim_nV^n=u(0)=\sfV(0)$.

Part 2 of the result is nothing but a reformulation of \eqref{b2},
using the definition of $\ph$ in \eqref{b1} and the definitions of
$\bar{\mathbb{H}}$ and $\mathbb{H}^{\ell n}$.
Part 3 refers to the construction of the control process $a^{*n}$
via $\psi^{\ell n}$ as done in \eqref{b3}, which we have just proved to be
AO for the adversary. This completes the proof.
\QED

\section*{Appendix}
\appendix

\noi{\bf Proof of Lemma \ref{lem:B is adapted}.} 
Fix $t\geq 0$ and $n$ and denote $m=\floornt$. 
By the definition of  $B^n$, $B^n_s=\beta^n(a^n_1,a^n_2,...)(s)$ for $0\leq s\leq t$.
Define $f(a^n_1,...,a^n_m)=\beta^n(a^n_1,...,a^n_m,a^{*n},a^{*n},a^{*n},...)(s)$ for $0\leq s\leq t$, where $a^{*n}=(\pi^{*n},J^{*n})$ and $\pi^{*n}$ is a fixed member of $\PI^n$ and $J^{*n}$ is a fixed member of $\R_+^L$.
The causality property \eqref{25}  implies $B^n_s=f(a_1^n,...,a^n_m)$ for all $s\in[0,t]$.
Therefore, for any measurable set $C$, since $\beta^n$ is a measurable map,
\begin{align*}
    \big\{\{B_s^n\}_{s\in[0,t]}\in C\big\}=\big\{f(a^n_1,...,a^n_m)\in C\big\}=\lt\{(a^n_k)_{k\in[m]}\in f^{-1}(C)\rt\}\in\sigma\big((J_k^{\ell n},\pi^{\ell n}_k),k\leq \floornt,\ell\in[L]\big).
\end{align*}
Because $\pi_k^{\ell n}$ is $\calF^n_{k-1}$-measurable, one has
$\sigma\big((J_k^{\ell n},\pi_k^{\ell n}),k\leq \floornt,\ell\in[L]\big)=\calF^n_{\floornt}=\calG^n_t$.
Therefore $\{B_t^n\}$ is $\{\calG_t^n\}$-adapted.
\null\hfill$\square$

\skp

\noi{\bf Acknowledgement.} RA is supported by ISF grant 1035/20.

\footnotesize

\bibliographystyle{is-abbrv}
\bibliography{Bib}

\end{document}